\documentclass[12pt, reqno]{amsart}

\usepackage{amsmath}
  \usepackage{paralist}
  \usepackage{graphics} 
  \usepackage{epsfig} 
\usepackage{graphicx}  \usepackage{epstopdf}
 \usepackage[colorlinks=true]{hyperref}
\hypersetup{urlcolor=blue, citecolor=red}

\newtheorem{theorem}{Theorem}[section]

\newtheorem{lemma}[theorem]{Lemma}
\newtheorem{proposition}{Proposition}

\theoremstyle{definition}
\newtheorem{definition}[theorem]{Definition}
\newtheorem{remark}{Remark}

\newtheorem{example}[theorem]{Example}

\newcommand{\J}{\mathcal{J}}

\newcommand{\norm}[1]{\lVert#1\rVert}
\newcommand{\R}{\mathbb{R}}
\newcommand{\N}{\mathbb{N}}
\newcommand{\M}{\mathcal{M}}

\newcommand{\mR}{\mathcal{R}}

\newcommand{\V}{\mathcal{V}}
\newcommand{\oOmega}{\overline{\Omega}}

\newcommand{\cat}{\text{cat}}
\newcommand{\dist}{\text{dist}}

\renewcommand{\d}{\delta}

\newcommand{\Mf}{\mathfrak{M}}
\newcommand{\Cf}{\mathfrak{C}}

\newcommand{\OFGC}{OFGC}
\newcommand{\OTFGC}{OTFGC}
\newcommand{\OGC}{OGC}

\newcommand{\sprod}[2]{\langle #1,#2\rangle}
\newcommand{\p}{\partial}

\usepackage{mathtools}
\mathtoolsset{showonlyrefs}

\title[
Multiplicity of OGC in Finsler disks]
{A Multiplicity Result for Orthogonal Geodesic Chords in Finsler disks}

\author[D. Corona]{}


 \email{dario.corona@unicam.it}

\begin{document}
\maketitle

\centerline{\scshape Dario Corona$^*$}
\medskip
{\footnotesize
 \centerline{Universit\`a di Camerino,}
   \centerline{Scuola di Scienze e Tecnologie }
   \centerline{ Camerino (MC), Italy}
} 


%
%

\begin{abstract}
	In this paper, we study the existence and multiplicity problems for orthogonal Finsler geodesic chords in a manifold with boundary which is homeomorphic to a $N$-dimensional disk.
	Under a suitable assumption, which is weaker than convexity, we prove that, if the Finsler metric is reversible, then there are at least $N$ orthogonal Finsler geodesic chords that are geometrically distinct.
	If the reversibility assumption does not hold, then there are at least two orthogonal Finsler geodesic chords with different values of the energy functional.
\end{abstract}

\section{Introduction}

Let $(\overline{\Omega},F)$ be a compact Finsler manifold of class $C^3$ with boundary $\p\Omega \in C^2$ and homeomorphic to an $N$-dimensional disk in $\R^N$, with $N \ge 2$.

\begin{definition}
	A curve $\gamma: [a,b] \to \overline\Omega$ is a \textit{Finsler geodesic chord} if 
	\begin{itemize}
		\item it is a geodesic with respect to the Finsler metric $F$, namely in local coordinates it satisfies the geodesic equations
		\begin{equation}
		\label{eqn:geodesicEquationsChern}
		\ddot{\gamma}^i + \Gamma_{jk}^i(\gamma,\dot{\gamma})\dot{\gamma}^j\dot{\gamma}^k
		= 0,
		\end{equation}
		where $\Gamma_{jk}^i$ are the components of the Chern connection;
		\item $\gamma(a),\gamma(b) \in \p\Omega$ and $\gamma(]a,b[) \subset \Omega$.
	\end{itemize}
	If $\gamma$ is also such that $\dot{\gamma}(a)$ and $\dot{\gamma}(b)$ are orthogonal to $T_{\gamma(a)}\p\Omega$ and $T_{\gamma(b)}\p\Omega$ respectively, i.e.
	\begin{equation}
	\label{eqn:conBounCond}
	\left.d_vF^2(\gamma(t),\dot{\gamma}(t))\right|_{T_{\gamma(t)}\p\Omega} = 0, \qquad t = a,b,
	\end{equation}
	then $\gamma$ is an \textit{orthogonal Finsler geodesic chord} (\OFGC).
\end{definition}

The aim of this article is to estimate the number of \OFGC s in $\overline\Omega$.
This question arises from the study of the brake-orbits in a potential well for a Hamiltonian system of classical type (cf. \cite{Weinstein1978}).
Indeed, using a Legendre transform and the Maupertuis-Jacobi principle, every brake orbit of a Hamiltonian system of classical type corresponds to a geodesic in a disk with endpoints on the boundary, where the disk is endowed with a Jacobi-Finsler metric.
When the Hamiltonian system is natural, namely is of the form
\begin{equation}
H(q,p) = \sum_{i,j = 1}^N g_{ij}(q)p^ip^j + V(q),
\end{equation}
then the associated Jacobi-Finsler metric turns out to be a Jacobi-Riemannian metric.
Seifert conjectured in \cite{Seifert1948} that there are at least $N$ brake orbits in an $N$-dimensional potential well of a natural Hamiltonian system.
This conjecture has been recently proved in \cite{Giambo2020Seifert}, exploiting also some partial results achieved by the same authors in different previous works (cf. \cite{Giambo2005,Giambo2009,Giambo2010,Giambo2011,Giambo2015,Giambo2018}).
In particular, in \cite{Giambo2005} every brake orbit of a natural Hamiltonian system is associated with an orthogonal geodesic chord in a regular and strictly concave domain.
Proving an equivalent result for a Hamiltonian system of classical type, every brake orbit in a potential well for such a kind of system could be seen as an \OFGC\ in a compact strictly concave Finsler manifold with boundary homeomorphic to an $N$-dimensional disk.

The first step of this kind of study is to use a sufficient condition for the existence of \OFGC:
the non existence of Finsler geodesic chords that start orthogonally and arrive tangentially to the boundary of $\Omega$.
More formally, we give the following definition.
\begin{definition}
	A curve $\gamma:[a,b] \to \overline\Omega$ is an \textit{orthogonal-tangent Finsler geodesic chord} (\OTFGC) if it is a Finsler geodesic chord, $\dot{\gamma}(a)$ is orthogonal to $\p\Omega$ and $\dot{\gamma}(b) \in T_{\gamma(b)}\p\Omega$.
\end{definition}

Now we are ready to state our main result.

\begin{theorem}
	\label{teo:mainTeo}
	Let $\overline{\Omega}\subset \M$ be an $N$-disk and $F:T\M \to [0,+\infty)$ a Finsler metric on $\M$.
	Then either:
	\begin{itemize}
		\item there exists an orthogonal-tangent Finsler geodesic chord
	\end{itemize}
	or
	\begin{itemize}
		\item if $F$ is a reversible Finsler metric, then there exist at least $N$ geometrically distinct \OFGC s; if $F$ is not reversible, then there are at least two \OFGC s with different values of the Finsler energy functional.
	\end{itemize}
\end{theorem}

As in \cite{Giambo2018} for the Riemannian case, we will follow a variational approach, seeing the geodesics as critical points of the Finsler energy functional
$$
\J(\gamma) = \frac{1}{2}\int_{0}^1 F^2(\gamma,\dot{\gamma}) ds
$$
defined on the set of $H^{1,2}$-curves that lay on $\overline{\Omega}$ and with endpoints in $\p\Omega$.
Indeed, the critical curves of $\J$ that do not touch the boundary in $(0,1)$ are Finsler geodesic chords 
and their existence and multiplicity can be obtained exploiting a suitable Ljusternik and Schnirelmann category.
However, due to the presence of the boundary, we have to treat with critical curves of $\J$ touching the boundary and with their regularity.
In the Riemannian case, the techniques presented in \cite{MarinoScolozzi83} could be applied to obtain the desired regularity (cf. \cite{Giambo2018}).
The same techniques could be applied in the Finsler setting requiring that
if $\nu:\p\Omega \to T\M$ is a unit normal vector field along $\p\Omega$, then for all $q \in \p\Omega$ 
\begin{equation}
\label{eqn:LgeomAss}
d_vF^2(q,\xi)[\nu] = 0, \qquad \forall \xi \in T_q\p\Omega.
\end{equation}
Under this assumption, a result weaker than Theorem \ref{teo:mainTeo} has been obtained in \cite{Corona2020}.
However, \eqref{eqn:LgeomAss} does not hold for the energy function  of a general Finsler metric (see Example \ref{ex:FinslerOrthogonality}).
Consequently, we shall use a penalization argument to prove the regularity of the Finsler geodesics in a manifold with boundary, following the approach presented in \cite{Bartolo2011} in the convex case.

\section{Framework Setup and Notation}
For the sake of presentation, we suppose that $\overline{\Omega}$ is embedded into a $N$-manifold $\M$ including $\overline{\Omega}$.
Using the Whitney embedding theorem, we can see $\M$ as a smooth ($C^3$) submanifold of $\R^{2N}$, endowed with the Riemannian structure of the euclidean scalar product $\sprod{\cdot}{\cdot}$ of $\R^{2N}$.
A coordinate system $(q^i) = (q^1,\dots, q^n)$ on $\M$ naturally induces a coordinate system $(q^i, v^i), i = 1,\dots, N$ on the tangent bundle $T\M$.
If $f$ is a real-valued function defined on $T\M$, then $d_q f$ and $d_v f$ will denote the derivatives of $f$ with respect to $q$ and $v$ respectively.
In a local chart, the derivatives with respect to $q^i$ and $v^i$ will be denoted by $\p_{q^i}$ and $\p_{v^i}$.
We will use the Einstein notation, implying summation over a set of indexed terms in a formula.
The norm $\norm{\cdot}:T\M \to \R$ is that one induced by the euclidean product in $\R^{2N}$, while we denote by $\norm{\cdot}_{L^p}$ the norm in a $L^p$ space, for any $1 \le p \le  \infty$.

\subsection{On Finsler metric}
We consider a Finsler metric on $\M$, namely a function $F:T\M \to \R$ that satisfies the following properties:
\begin{itemize}
	\item it is continuous on $T\M$ and $C^3$ on $T\M \setminus0$;
	\item it is fiberwise positively homogeneous of degree one, namely
	$$
	F(q,\lambda v)  = \lambda F(q,v) \quad \forall  \lambda > 0;
	$$
	\item it is strictly positive unless $v \ne 0$;
	\item for all $(q,v) \in T\M \setminus 0$, the symmetric bilinear form $g_v:T_q\M \times T_q\M \to \R$
	\begin{equation}
	g_v(\xi, \eta) := \frac{1}{2}\frac{\partial^2 }{\p s \p t} \left[F^2(q, v+ s\xi + t \eta)\right]_{s = t = 0}
	\end{equation}
	is positive definite.
	In local coordinates, this means that the matrix
	\begin{equation}
	\label{eqn:gijDef}
	g_{ij}(q,v) = \frac{1}{2}\partial^2_{v^i v^j} F^2(q,v)
	\end{equation}
	is positive definite.
\end{itemize}
Moreover, if $F(q,v) = F(q,-v)$ for all $q \in \M$ and $v \in T_q\M$,
we said that it is \textit{reversible}.
In the following we denote the function $F^2$ by $G$.

The interested reader can find more details about Finsler geometry in \cite{Shen2001}.

\subsection{Geometry of $\overline{\Omega}$}
\label{sec:geometryOmega}

There exists a function $\Phi: \M \to \R$ of class $C^2$
such that $\Omega = \Phi^{-1}(]-\infty, 0[)$, $\p\Omega = \Phi^{-1}(0)$ and $d\Phi(q) \ne 0$ for every $q \in \p\Omega$.
For any $\delta > 0$, we 
set
\begin{equation}
\Omega_\d = \Phi^{-1}(]-\infty,\delta[).
\end{equation}
By the $C^2$ regularity of $\Phi$, there exists a $\delta_0 > 0$ such that 
\begin{equation}
\label{eqn:delta0Def}
d\Phi(q) \ne 0, \quad \forall q \in \Phi^{-1}([-\delta_0,\delta_0]),
\end{equation}
and such that $\oOmega_\d$ is compact for any $\delta \in [0,\delta_0]$.
We also set 
\begin{equation}
\label{eqn:K0def}
K_0 = \max_{q \in \overline\Omega_{\d_0}} \norm{\nabla\Phi(q)}.
\end{equation}

%

\begin{remark}
	\label{rmk:aDef}
	Since $\overline{\Omega}_{\d_0}$ is compact, there exists a constant $\ell > 0$ such that the following inequalities hold for every $q \in \overline{\Omega}_{\d_0}$ and $v \in T_q\M$:
	\begin{equation}
	\label{eqn:boundG} 
	\frac{1}{\ell} \norm{v}^2 
	\le G(q,v)
	\le \ell \norm{v}^2;
	\end{equation}
	\begin{equation}
	\label{eqn:boundDqG}   
	\left\| d_qG(q,v) \right\| \le \ell \left(1 + \norm{v}^2\right)
	\quad \text{and} 
	\quad
	\left\| d_vG(q,v) \right\| \le \ell \left(1 + \norm{v}\right);
	\end{equation}
	\begin{equation}
	\label{eqn:boundDeriv2}
	\left\| d_{qq}G(q,v) \right\| \le \ell \left(1 + \norm{v}^2 \right)
	\quad \text{and} 
	\quad
	\left\| d_{qv}G(q,v) \right\| \le \ell \left(1 + \norm{v} \right);
	\end{equation}
	and if $v \ne 0$, then
	\begin{equation}
	\label{eqn:boundDeriv2vv}
	\left\| d_{vv}G(q,v) \right\| \le \ell.
	\end{equation}
	Moreover, there exist a constant $\alpha > 0$ such that
	\begin{equation}
	\label{eqn:defAlpha}
	d^2_{vv}G(q,v)[\xi, \xi] \ge \alpha \norm{\xi}^2,
	\qquad  
	\forall q \in \overline{\Omega}_{\d_0}, \forall v \in T_q\M \setminus 0. 
	\end{equation}
\end{remark}

\begin{lemma}
	\label{lem:barDelta}
	Set
	\begin{equation}
	\label{eqn:defBarDelta}
	\d_m = \frac{\delta_0^2}{2\ell K_0^2}.
	\end{equation}
	If $x(0) \in \p\Omega$ and $\J(x) < \d_m$, then 
	$$
	\max_{s \in [0,1]}|\Phi(x(s))| \le \delta_0,
	$$
	where $\delta_0$ were introduced in \eqref{eqn:delta0Def}.
\end{lemma}
\begin{proof}
	By \eqref{eqn:K0def} and \eqref{eqn:boundG} and using the Cauchy-Schwarz inequality, we have the following chain of inequalities
	\begin{multline}
	|\Phi(x(s))| = |\Phi(x(s)) - \Phi(x(0))| 
	\le \int_{0}^{s} |\sprod{\nabla\Phi(x)}{\dot{x}}|ds
	\le K_0 \int_{0}^s \norm{\dot{x}} ds \\ 
	\le K_0 
	\left(\int_{0}^1 \norm{\dot{x}}^2 ds\right)^{\frac{1}{2}}
	\le K_0 
	\left(\ell \int_{0}^1 G(x,\dot{x}) ds\right)^{\frac{1}{2}} =
	K_0 \sqrt{2\ell \J(x)}
	\le \delta_0.
	\end{multline}
\end{proof}

\begin{definition}
	For every $q \in \M$ and $v \in T_q\M \setminus 0$, the Finsler Hessian of $\Phi$ at $(q,v)$ is defined as 
	\begin{equation}
	\label{eqn:DefFinHes}
	H_{\Phi}(q,v)[v,v] = \frac{d^2}{ds^2}(\Phi \circ \gamma)(0),
	\end{equation}
	where $\gamma$ is the geodesic of $(\M,F)$ such that $\gamma(0) = q$ and $\dot\gamma(0) = v$.
\end{definition}
Using the geodesic equations \eqref{eqn:geodesicEquationsChern}, in local coordinates $H_\Phi(q,v)[v,v]$ is given by 
\begin{equation}
\label{eqn:FinHesLocCoo}
(H_\Phi)_{ij}(q,v)v^i v^j = 
\p^2_{q^i q^j}\Phi(q) v^i v^j - 
\p_{q^k}\Phi(q)\Gamma_{ij}^k(q,v)v^i v^j.
\end{equation}

\subsection{Sobolev and functional spaces}
We consider the Sobolev spaces
$
H^{1,2}([0,1],\R^{2N})
$
and
$$
H^{1,2}_0([0,1],\R^{2N}) = 
\left\{
V \in H^{1,2}\left([0,1],\R^{2N}\right): V(0) = V(1) = 0
\right\}.
$$
	For $ S \subset \M$, set
	\begin{equation}
	H^{1,2}([0,1],S) = 
	\left\{
	x \in H^{1}([0,1],\R^{2N}): x(s) \in S \text{ for all } s \in [0,1]
	\right\}.
	\end{equation}
	It is well known that 
	$H^{1}([0,1],\M)$ 
	is a manifold of class $C^2$ and its tangent space at $x$ is
	\begin{equation}
	T_x\M = 
	\left\{
	\xi \in H^{1,2}([0,1], \R^{2N}): \xi(s) \in T_{x(s)}\M \text{ for all } s \in [0,1]
	\right\}.
	\end{equation}

Due to the presence of the boundary $\p\Omega$, not all the elements of $T_x\M$ are always admissible variations, according to the subset of $ H^{1,2}([0,1],\M)$ we are considering.
So we give the following definition.
\begin{definition}
	Let $\mathcal{Q}$ be a non-empty subset of $H^{1,2}([0,1],\M)$
	.
	Then $\xi \in T_x\M$ is an \textit{admissible infinitesimal variation of $x$ in $\mathcal{Q}$} if there exists an $\epsilon > 0$ and a differentiable function $h:(-\epsilon,\epsilon) \times [0,1] \to \M$ such that
	\begin{itemize}
		\item $h(0,s) = x(s)$;
		\item $h(\tau,\cdot) \in \mathcal{Q}$ for all $\tau \in (-\epsilon,\epsilon)$;
		\item $\left.\frac{\p h}{\p \tau}(\tau,s)\right|_{\tau = 0} = \xi(s)$.
	\end{itemize}
	The set of all admissible infinitesimal variation of $x$ in $\mathcal{Q}$ is denoted by $\V^-(x,\mathcal{Q})$.
\end{definition}

\begin{definition}
	Let $\mathcal{Q}$ be a non-empty subset of $H^{1,2}([0,1],\M)$ and $\mathcal{F}$ a functional of class $C^1$ defined on $\mathcal{Q}$.
	A curve $x \in \mathcal{Q}$ is said $\V^-$-\textit{critical curve for $\mathcal{F}$ on $\mathcal{Q}$} if 
	\begin{equation}
	d\mathcal{F}(x)[\xi] \ge 0, \qquad \forall \xi \in \V^-(x,\mathcal{Q}).
	\end{equation}
\end{definition}

The main functional space of our variational problem is 
\begin{equation}
\Mf = \{x \in H^{1,2}([0,1], \overline{\Omega}): x(0), x(1) \in \p\Omega\}.
\end{equation}
If $x \in \Mf$, then
\begin{equation}
\V^-(x,\Mf) = 
\left\{
\begin{aligned}
\xi  & \in T_x\M: \: 
\sprod{\nabla \Phi(x(0))}{\xi(0)}  = \sprod{\nabla \Phi(x(1))}{\xi(1) } = 0, \\
& \sprod{\nabla \Phi(x(s))}{\xi(s)} \le 0 \text{ for any }s\in (0,1)\text{ such that } x(s) \in \p\Omega
\end{aligned}
\right\}.	
\end{equation}
In other words, a vector field $\xi \in T_x\M$ is in $\V^-(x,\Mf)$ if $\xi(0)$ and $\xi(1)$ are tangent to $\p\Omega$ and $\xi(s)$ points inside $\overline\Omega$ whenever $x(s) \in \p\Omega$.

In $H^{1,2}([0,1],\R^{2N})$ we define the norm $\norm{\cdot}_*$ as
\begin{equation}
\label{eqn:defNorm*}
\norm{\xi}_* = \max\left\{\norm{\xi(0)},\norm{\xi(1)}\right\} + 
\left(\int_{0}^1 \lVert{\dot\xi}\rVert^2ds\right)^{\frac{1}{2}},
\end{equation}
while in $H^{1,2}([0,1],\M)$ we define the distance function
\begin{multline}
\label{eqn:defDist*}
\dist_*(x_1,x_2) = 
\max\left\{
\norm{x_1(0)-x_2(0)}, \norm{x_1(1)-x_2(0)}
\right\} \\
+ 
\left(
\int_{0}^{1}
\norm{\dot{x}_1(s) - \dot{x}_2(s)}^2 \ ds
\right)^{\frac{1}{2}}.
\end{multline}

\subsection{The Finsler energy functional}


%
%

We consider on $H^{1,2}([0,1],\M)$ the Finsler energy functional
\begin{equation}
\label{eqn:Jdef}
\J(\gamma) = \frac{1}{2} \int_0^1 G(\gamma,\dot{\gamma}),
\end{equation}
whose differential is
\begin{equation}
\label{eqn:Jdiff}
d\J(\gamma)[\xi] =
\frac{1}{2} \int_{0}^{1}
\left(
d_qG(\gamma,\dot{\gamma})[\xi] + 
d_vG(\gamma,\dot{\gamma})[\dot{\xi}]
\right)ds
, \qquad \forall \xi \in T_x\M.
\end{equation}
For our purposes, we will consider the restrictions of $\J$ to some subsets of $H^{1,2}([0,1],\M)$, which will be denoted again with $\J$.

Set 
\begin{equation}
\mathcal{C}(p,q,\M) = \left\{
x \in H^{1}([0,1],\M): x(0) = p, x(1) = q
\right\}.
\end{equation}
Then a curve $\gamma \in \mathcal{C}(p,q,\M)$ is a Finsler geodesic parametrized with constant speed if and only if it is a critical point for $\J$ on $\mathcal{C}(p,q,\M)$.
In local coordinates, a geodesic satisfies the equations
\begin{equation}
\label{eqn:geodesicEquationsPre}
\ddot{\gamma}^i + \frac{1}{2} g^{ij}(\gamma,\dot{\gamma})
\left(
\p^2_{q^k v^j} G(\gamma,\dot{\gamma})\dot{\gamma}^k -
\p_{q^j} G(\gamma,\dot{\gamma})
\right) = 0,
\end{equation}
where $g^{ij}$ are the components of the inverse matrix of the fundamental tensor \eqref{eqn:gijDef}.
Denoting by $\Gamma_{jk}^i$ the components of the Chern connection (see \cite[Chapter 5]{Shen2001}), the previous equations can be written as in
\eqref{eqn:geodesicEquationsChern}.

\subsection{Backward parametrization map}
We define the backward parametrization map
$\mR:\Mf \to \Mf$ by
\begin{equation}
(\mR x)(s) = x(1 - s), \qquad \forall s \in [0,1].
\end{equation}

We say that $\mathcal{N} \subset \Mf$ is $\mathcal{R}$-invariant if $\mathcal{R}(\mathcal{N}) = \mathcal{N}$.
On any $\mR$-invariant set $\mathcal{N}$, the backward parametrization map $\mathcal{R}$ induces an equivalence relation and we denote by $\widetilde{\mathcal{N}}$ the quotient space.
Through this equivalence relation, we identify any element $x$ of $\Mf$ with its backward parametrization.

The map $\mR$ induces a map $\overline\mR:\V^-(x,\Mf) \to \V^-(\mR x,\Mf)$ defined by
\begin{equation}
\label{eqn:defBarmR}
(\overline\mR \xi)(s) = \xi(1- s).
\end{equation}

\begin{lemma}
	\label{lem:RxCritical}
	Let $F$ be a reversible Finsler metric.
	If $x$ is a $\V^-$-critical curve for $\J$ on $\Mf$, then also $\mR x$ is a $\V^-$-critical curve for $\J$ on $\Mf$.
\end{lemma}
\begin{proof}
	Set $y = \mR x$ and for any $\zeta \in \V^-(y,\Mf)$, set $\xi = \overline\mR\zeta \in \V^-(x,\Mf)$.
	Then, setting $t = 1 - s$, we obtain
	\begin{multline}
	d\J(x)[\xi] = 
	\frac{1}{2} \int_{0}^{1}
	\left(
	d_qG(x,\dot{x})[\xi] + 
	d_vG(x,\dot{x})[\dot{\xi}]
	\right)ds
	\\
	=
	\frac{1}{2} \int_{0}^{1}
	\left(
	d_qG(y,-\dot{y})[\zeta] + 
	d_vG(y,-\dot{y})[-\dot{\zeta}]
	\right)dt 
	\\
	=
	\frac{1}{2} \int_{0}^{1}
	\left(
	d_qG(y,\dot{y})[\zeta] + 
	d_vG(y,\dot{y})[\dot{\zeta}]
	\right)dt 
	= d\J(y)[\zeta].
	\end{multline}
	Consequently, if $d\J(x)[\xi] \ge 0$ for all $\xi \in \V^-(x,\Mf)$, then also $d\J(y)[\zeta] \ge 0$ for all $\zeta \in \V^-(y,\Mf)$, so $y = \mR x$ is a $\V^-$-critical curve for $\J$ on $\Mf$.
\end{proof}

\subsection{Ljusternik and Schnirelmann relative category}
\label{sec:relativeCategory}
\begin{definition}
	Let $X$ be a topological space and $Y$ a closed subset of $X$.
	A closed subset $F$ of $X$ has \textit{relative category} equal to $k \in \N$,
	\begin{equation}
	\text{cat}_{X,Y}(F) = k,
	\end{equation}
	if $k$ is the minimal positive integer such that
	$F \subset \bigcup_{i = 0}^k A_i$, where $\{A_i\}_{i = 0}^k$ is a family of open subset of ${X}$ satisfying:
	\begin{itemize}
		\item ${F} \cap {Y} \subset A_0$;
		\item if $i \ne 0$, $A_i$ is contractible in $X\setminus Y$;
		\item if $i = 0$, there exists $h_0 \in C^0([0,1]\times A_0, {X})$
		such that $h_0(1,A_0) \subset {Y}$ and $h_0 ([0,1], A_0\cap{Y}) \subset {Y}$.
	\end{itemize}
\end{definition}

The following lemma describes an $\mathcal{R}$-invariant subset $\Cf$ of $\Mf$ 
whose relative category will play a central role in the proof of our multiplicity result.
\begin{lemma}
	\label{lemma:chordsDef}
	There exists a continuous map $\gamma: \p\Omega \times \p\Omega \to \Mf$ such that
	\begin{itemize}
		\item $\gamma(A,B)(0) = A, \gamma(A,B)(1) = B$;
		\item $A\ne B \implies \gamma(A,B)(s) \in \Omega, \; \forall s \in ]0,1[$;
		\item $\gamma(A,A)(s) = A, \; \forall s \in [0,1]$;
		\item $\mathcal{R}\gamma(A,B) = \gamma(B,A)$.
	\end{itemize}
\end{lemma}
\begin{proof}
	Since $\overline{\Omega}$ is a Riemannian $N$-disk, there exists an homeomorphism $\Psi: \overline{\Omega} \to \mathbb{D}^N$.
	We define 
	\begin{equation}
	\tilde{\gamma}(A,B)(s) = \Psi^{-1}\left((1-s)\Psi(A) + s\Psi(B)\right) \quad \forall A,B \in \p\Omega.
	\end{equation}
	The construction above produces curves that are a priori only continuous.
	In order to produce curves with an $H^{1,2}$-regularity, it suffices to use a broken geodesic approximation argument.
\end{proof}
After choosing a function $\gamma:\p\Omega \times \p\Omega \to \Mf$ as in Lemma \ref{lemma:chordsDef}, we define
\begin{equation}
\label{eqn:Def-Cf-Cf0}
\begin{split}
& \Cf = \{\gamma(A,B): A,B \in \p\Omega\}, \\
&\Cf_0 = \{\gamma(A,A): A \in \p\Omega\}.
\end{split}
\end{equation}
In particular, $\Cf_0$ is the set of constant paths on the boundary $\p\Omega$.
The multiplicity of \OFGC s for reversible Finsler metrics is based on the relative category
\begin{equation}
\label{eqn:relCatReversible}
\cat_{\tilde\Cf,\tilde\Cf_0}\tilde{\Cf} \ge N,
\end{equation}
while for the non-reversible case we exploit the following inequality
\begin{equation}
\label{eqn:relCatNonRev}
\cat_{\Cf,\Cf_0}{\Cf} \ge 2.
\end{equation}
The inequalities \eqref{eqn:relCatReversible} and \eqref{eqn:relCatNonRev} have been proved in \cite{Giambo2010} and \cite{Giambo2015} respectively.

For our porpuses, we need to define the following quantity
\begin{equation}
\label{eqn:deltaM}
\d_M= \max_{x \in \mathfrak{C}}\J(x).
\end{equation}

\section{Regularity of the $\V^-$-critical curves for $\J$ on $\Mf$}

The main goal of this section is to prove that either every non-constant $\V^-$-critical curve for $\J$ on $\Mf$ is an \OFGC\ or there exists an \OTFGC.
Towards this aim, we need to provide the following regularity result.
\begin{proposition}
	\label{prop:regularity}
	Let $x\in \Mf$ be a $\V^-$-critical curve for $\J$ on $\Mf$.
	Then $x$ has $H^{2,\infty}$-regularity, namely $\dot{x}$ is absolutely continuous and $\ddot{x} \in L^\infty([0,1],\R^N)$.
	Moreover, setting $C_x = \{s \in [0,1]: x(s) \in \p\Omega \}$,
	there exists a function $\lambda \in L^\infty([0,1],\R)$ such that 
	\begin{equation}
	\label{eqn:ELboundary}
	d_qG(x, \dot{x}) - \frac{d}{ds}d_vG(x, \dot{x}) = \lambda \nabla\Phi(x)
	\end{equation}
	holds almost everywhere, 
	$\lambda \le 0$ a.e. in $[0,1]$, $\lambda(s) = 0$ if $s \notin C_x$ and
	\begin{equation}
	\label{eqn:lambdaCx}
	\lambda = 
	\frac{
		H_\Phi(x,\dot{x})[\dot{x},\dot{x}]
	}{
		g^{ij}(x,\dot{x})\p_{q^i}\Phi(x)\p_{q^j}\Phi(x)
	}
	\qquad \text{a.e. in } C_x.	
	\end{equation}
\end{proposition}

The previous regularity result can be proved in the Riemannian case using the techniques introduced in \cite{MarinoScolozzi83}.
That kind of proof relies on the existence of a non-zero vector field $\mu$ on $\p\Omega$ such that $g(\xi,\mu) = 0$ for all $\xi \in T\p\Omega$, where $g$ is the Riemannian metric.
However, that method cannot be exploited in the Finsler case, since the orthogonality condition between two vectors $\eta,\xi \in T_q\M$, given by
\begin{equation}
\label{eqn:FinslerOrthogonality}
d_vF^2(q,\eta)[\xi] = 
g_\eta(\eta,\xi) = 0,
\end{equation}
is not symmetric.
For this reason, there could not exist a vector field $\mu$ in $\p\Omega$ such that $\mu(q) \ne 0$ and $d_v G(q,\xi)[\mu] = 0$ for all $q \in \p\Omega$ and $\xi \in T_q\p\Omega$, as shown by the following example.
\begin{example}
	\label{ex:FinslerOrthogonality}
	Set $\Omega = \left\{z \in \R^3: \norm{z} < 1 \right\}$ and 
	$$
	F(z,v) = \left((v^1)^p + (v^2)^p + (v^3)^p\right)^{\frac{1}{p}},
	$$
	for $p > 2$.
	In local coordinates, the orthogonality condition $d_v G(z,\xi)[\mu] = 0$ reads as 
	$$
	\frac{2 (p-1)}{p} F^{2 - p}(\xi)
	\left(
	(\xi^1)^{p-1}\mu^1 + (\xi^2)^{p-1}\mu^2 + (\xi^3)^{p-1}\mu^3
	\right) = 0,
	$$
	which is not linear with respect to $\xi$.
	Consequently, there could exists some $z \in \p\Omega$ such that $d_v G(z,\xi)[\mu] = 0$ for all $\xi \in T_z\p\Omega$ if and only if $\mu(z) = 0$.
	For instance, set $p = 3$ and $z = (1/\sqrt{3}, 1/\sqrt{3}, 1/\sqrt{3})$.
	Then $\xi \in T_z\p\Omega$ if and only if $\xi^3 = - \xi^1 - \xi^2$, the orthogonality condition is 
	$$
	(\xi^1)^2 \mu^1 + (\xi^2)^2 \mu^2 + ((\xi^1) + (\xi^2))^2 \mu^3 = 0,
	$$
	and it is easy to check that the only $\mu$ that satisfies this equation for every $\xi \in T_z\p\Omega$ is the zero vector.
\end{example}

Since we cannot use the techniques presented in \cite{MarinoScolozzi83}, we should employ other methods to prove the regularity of the $\V^-$-critical curves.
In \cite{Canino1991}, the regularity of Euler-Lagrange orbits for a general Tonelli-Lagrangian of class $C^2$ in a compact manifold with boundary is achieved applying directly the definition of critical curve for a nonsmooth functional, which in our setting is equivalent to the definition of $\V^-$-critical curve.
Because of its technicality, we will not follow that approach and we will prove Proposition \ref{prop:regularity} using a \textit{penalization method}.
Since the critical curves of the penalized functional must lay on the interior of an open set, their regularity can be proved by a standard argument.
Then we can prove the regularity of the $\V^-$-critical curves of the functional taking the limit to remove the penalization term.
The penalization method in a manifold with boundary has been exploited, for instance, in \cite{GiannoniMajer1997}, in \cite{GiannoniMasiello1991} and in \cite{Bartolo2011} for the Riemannian, Lorentzian and Finsler case respectively.
In particular, the regularity of Finsler geodesics in a domain with boundary is proved in \cite{Bartolo2011} assuming the domain to be strictly convex.
Since we do not require this assumption, we will exploit the uniqueness of the local minimum of the energy functional defined on the set of curves with fixed endpoints.

For any $[a,b] \subset [0,1]$, we define the functional 
$$
\J^{a,b}: H^{1,2}([0,1],\M) \to \R
\quad
\text{ by }
\quad
\J^{a,b}(x) = \frac{1}{2}\int_{a}^{b} G(x,\dot{x})\ ds,
$$
and we set
\begin{equation}
\mathcal{C}([a,b],p,q,\Omega_\d) =
\left\{
\gamma \in H^{1,2}([a,b],\Omega_\d): \gamma(a) = p,\ \gamma(b) = q
\right\},
\end{equation}
for any $p,q \in \oOmega$ and $\d \le \d_0$, where $\d_0$ has been defined in \eqref{eqn:delta0Def}.
For the sake of presentation, we denote $\mathcal{C}([a,b],p,q,\Omega_\d)$ by $\mathcal{C}_\d$ when we have fixed $[a,b]$ and $p,q \in \Omega$ and no confusion may arise.
We consider on $\mathcal{C}([a,b],p,q,\Omega_\d)$ the penalized functional
\begin{equation}
\J_{\d}(\gamma) = \J^{a,b}(\gamma)
+ \int_{a}^b \chi_\d(\Phi(\gamma))\ ds
,
\end{equation}
where the function $\chi_\d: ]-\infty,\d[ \to \R$ is defined by
\begin{equation}
\label{eqn:defChi}
\chi_\d(t) = 
\begin{cases}
0 & \mbox{if } t \le 0, \\
\frac{t^2}{(\d - t)^2} & \mbox{if } 0 \le t < \d.
\end{cases}
\end{equation}
By definition of $\chi_\d$ we have that
\begin{equation}
\label{eqn:chiDeriv}
\chi'_\d(t) = \frac{2\d}{t(\d- t)}\chi_\d(t).
\end{equation} 
The following lemma, known as Gordon's lemma (cf. \cite{Gordon1974}), stands at the core of the penalization method, because it allows to prove that $\J_\d$ attains its minimum on $\mathcal{C}_\d$.
\begin{lemma}[Gordon's lemma]
	\label{lem:gordon}
	Let $(\gamma_n)_n$ be a sequence in
	$\mathcal{C}([a,b],p,q,\Omega_\d)$ such that
	\begin{equation}
	\int_{a}^{b} G(\gamma_n,\dot\gamma_n) ds \le k < +\infty, \qquad \forall n\in \N,
	\end{equation}
	for some $k > 0$.
	If there is a sequence $(s_n)_n$ in $[a,b]$ such that 
	\begin{equation}
	\lim_{n \to \infty}\Phi(\gamma_n(s_n)) = \d, 
	\end{equation}
	then
	\begin{equation}
	\lim_{n \to \infty} \int_{a}^{b}
	\chi_\d(\Phi(\gamma_n)) ds = + \infty.
	\end{equation}
\end{lemma} 
\begin{proof}
	Recalling the definition of $K_0$ and $\ell$ in \eqref{eqn:K0def} and \eqref{eqn:boundG} respectively, for any $s \in [a,s_n]$ we have 
	\begin{multline}
	\label{eqn:gordonProof1}
	\Phi(\gamma_n(s_n)) - \Phi(\gamma_n(s)) = 
	\int_{s}^{s_n}
	\sprod{\nabla\Phi(\sigma)}{\dot\gamma_n(\sigma)} d\sigma
	\le
	\int_{s}^{s_n}
	K_0 \norm{\dot\gamma_n(\sigma)} d\sigma
	\\
	\le
	K_0 (s_n - s)^{\frac{1}{2}}
	\left(
	\ell
	\int_{s}^{s_n}
	G(\gamma_n(\sigma),\dot\gamma_n(\sigma))d\sigma
	\right)^{\frac{1}{2}}
	\le C (s_n - s)^{\frac{1}{2}},
	\end{multline}
	for some strictly positive constant $C$ that depends on $\ell$, $K_0$ and $k$, but not on $n$.
	Then
	$$
	0 < \d - \Phi(\gamma_n(s)) 
	\le
	C(s_n - s)^{\frac{1}{2}}
	+ \left(\d -  \Phi(\gamma_n(s_n)) \right),
	$$
	and
	\begin{equation}
	\label{eqn:gordonProof2}
	\frac{1}{\left( \d - \Phi(\gamma_n(s))\right)^2 }
	\ge
	\frac{1}{
		2\left(C^2 (s_n - s) + \left( \d - \Phi(\gamma_n(s_n)\right)^2 \right).
	}
	\end{equation}
	
	Since $\Phi(\gamma_n(s_n))\to \d > 0$, for $n$ sufficiently large $\displaystyle \Phi(\gamma_n(s_n)) > \frac{2}{3}\d$ and there exists a sequence $\bar{s}_n < s_n$ such that 
	\begin{equation}
	\Phi(\gamma_n(\bar{s}_n)) = \frac{1}{3} \d.
	\end{equation}
	From 
	\eqref{eqn:gordonProof1} we get that
	$$
	\left(s_n - \bar{s}_n\right)^\frac{1}{2} \ge \frac{1}{3C}\d > 0.
	$$
	Clearly we can choose $\bar{s}_n$ such that
	$$
	\Phi(\gamma_n(s)) > \frac{1}{3} \d, \qquad \forall s \in (\bar{s}_n , s_n).
	$$
	Thus, integrating both hands sides of \eqref{eqn:gordonProof2} we obtain
	\begin{equation}
	\int_{a}^{b} \chi_\d(\gamma_n(s)) \ ds \ge
	\frac{1}{9}
	\int_{\bar{s}_n}^{s_n} 
	\frac{\d^2}{
		2\left(C^2 (s_n - s) + \left( \d - \Phi(\gamma_n(s_n)\right)^2 \right)
	}
	ds
	\end{equation}
	and passing to the limit we get the thesis.
\end{proof}

\begin{lemma}
	\label{lem:Jdminimum}
	For any $\d \in (0,\d_0)$, $\J_{\d}$ has a minimum on $\mathcal{C}_\d$.
\end{lemma}
\begin{proof}
	Since $\J_\d$ is bounded from below, there exists a sequence $(\gamma_n)_n$ in $\mathcal{C}_\d$ such that
	$$
	\inf_{\gamma \in \mathcal{C}_\d}\J_\d(\gamma) 
	= \lim_{n \to \infty} \J_\d(\gamma_n).
	$$
	As a consequence, there exists a constant $k$ such that
	\begin{equation}
	\label{eqn:JdminProof1}
	\frac{1}{2}\int_{a}^{b} G(\gamma_n,\dot\gamma_n)ds 
	\le \J_\d(\gamma_n) 
	\le k,
	\end{equation}
	for $n$ sufficiently large.
	The bounds in \eqref{eqn:boundG} imply that
	$$
	\norm{\dot{\gamma}_n}_2^2 
	= \int_{a}^b \norm{\dot{\gamma}_n}^2 ds 
	\le \ell \int_{a}^{b}
	G(\gamma_n,\dot{\gamma}_n) ds 
	\le 2\ell k,
	$$
	so we can apply the Ascoli-Arzel\'a theorem and obtain a subsequence, which we denote by $(\gamma_n)$ again, that uniformly converges to a curve $\bar\gamma$ and such that $\dot{\gamma}_n$ weakly converges to $\dot{\bar\gamma}$.
	To conclude the proof, we will show that 
	$$ 
	\inf_{\gamma \in \mathcal{C}_\d} \J_\d(\gamma)
	\ge
	\J_\d(\bar\gamma).
	$$
	By Lemma \ref{lem:gordon} and \eqref{eqn:JdminProof1}, there exists a strictly positive constant $c$ such that $\d - \Phi(\gamma(s)) > c$ for any $s \in [a,b]$.
	Hence, $\bar{\gamma} \in \mathcal{C}_\d$ and 
	\begin{equation}
	\label{eqn:JdminProof2}
	\lim_{n \to \infty} \int_a^b \chi_\d(\Phi(\gamma_n)) ds
	= 
	\int_a^b
	\chi_\d(\Phi(\bar\gamma))ds.
	\end{equation}
	The uniform convergence of $\gamma_n$ to $\gamma$ leads to 
	\begin{equation}
	\label{eqn:JdminProof3}
	\lim_{n \to \infty}
	\int_{a}^{b} G(\gamma_n,\dot\gamma_n) ds
	= 
	\lim_{n \to \infty}
	\int_{a}^{b} G(\bar\gamma,\dot\gamma_n) ds.
	\end{equation}
	Moreover, since we are on a compact subset of $\M$
	\begin{equation}
	\xi \in L^2([a,b],\R^N) \to \int_{a}^{b} G(\bar\gamma(s),\xi(s)) ds
	\end{equation}
	is continuous with respect the strong topology of $L^2$.
	By the convexity of $G$, the above functional is convex.
	As a consequence, it is also lower semicontinuous with respect to the weak topology 
	and since $\dot\gamma_n$ weakly converges to $\dot{\bar\gamma}$ we obtain that
	\begin{equation}
	\label{eqn:JdminProof4}
	\liminf_{n \to \infty} 
	\int_{a}^{b} G(\bar\gamma,\dot\gamma_n) ds
	\ge 
	\int_{a}^{b} G(\bar\gamma,\dot{\bar\gamma}) ds.
	\end{equation}
	Finally, by 
	\eqref{eqn:JdminProof2},
	\eqref{eqn:JdminProof3} and
	\eqref{eqn:JdminProof4}
	we obtain
	\begin{multline}
	\inf_{\gamma \in \mathcal{C}_\d}J_\d(\gamma) 
	= \lim_{n \to \infty} \J_\d(\gamma_n)
	= \lim_{n \to \infty} 
	\left(
	\frac{1}{2}\int_{a}^{b} G(\gamma_n,\dot\gamma_n)ds 
	+
	\int_a^b \chi_\d(\Phi(\gamma_n)) ds
	\right)
	\\
	= \liminf_{n \to \infty}
	\left(
	\frac{1}{2}\int_{a}^{b} G(\bar\gamma,\dot{\gamma}_n)ds 
	\right)
	+
	\int_a^b \chi_\d(\Phi(\bar\gamma)) ds
	\\
	\ge
	\frac{1}{2}\int_{a}^{b} G(\bar\gamma,\dot{\bar\gamma})ds 
	+
	\int_a^b \chi_\d(\Phi(\bar\gamma))ds
	= \J_\d(\bar\gamma).
	\end{multline}
\end{proof}

\begin{lemma}
	\label{lem:regCritJd}
	For any ${\d \in (0,\d_0)}$, let $\gamma_\d$ be a critical point of $\J_\d$ in $\mathcal{C}_\d$.
	Then $\gamma_\d$ is $C^1$ and, for any $\bar{s} \in [a,b]$ such that $\dot\gamma(\bar{s}) \ne 0$, there exists a neighbourhood $I$ of $\bar{s}$ such that $\gamma_\d$ is twice differentiable in $I$ and there it satisfies the equations
	\begin{equation}
	\label{eqn:eqGeoPen}
	\ddot\gamma_\d^i + \Gamma_{jk}^i(\gamma_{\d},\dot\gamma_\d)
	\dot\gamma_\d^j \dot\gamma_\d^k =
	\chi_\d'(\Phi(\gamma))
	\p_{q^k}\Phi(\gamma_\d)
	g^{ki}(\gamma_{\d},\dot\gamma_\d).
	\end{equation}
	Moreover, a constant $E_\d \in \R$ exists such that
	\begin{equation}
	\label{eqn:conEnerDelta}
	E_\d = \frac{1}{2}G(\gamma_\d,\dot\gamma_\d) - \chi_\d(\Phi(\gamma_\d))
	\quad\text{ on }[a,b].
	\end{equation}
\end{lemma}
The proof of this lemma can be found in \cite[Lemma 4.1]{Bartolo2011}.
\begin{lemma}
	\label{lem:Jdlimitato}
	For any $\d \in (0,\d_0)$, let $\gamma_\d$ be a minimum of $\J_\d$ on $\mathcal{C}_\d$.
	Then there exists a constant $k_1 > 0$ such that 
	\begin{equation}
	\label{eqn:Jdlimitato}
	\J_\d(\gamma_\d) \le k_1 < +\infty ,
	\quad \forall \d \in (0,\d_0),
	\end{equation}
	and
	\begin{equation}
	\label{eqn:Gdbounded}
	\frac{1}{2} G(\gamma_\d(s),\dot\gamma_\d(s)) \le \frac{k_1}{b-a} + \chi_\d(\Phi(\gamma_\d(s)))
	\quad
	\text{ for all $\d \in (0,\d_0)$ and $s \in [a,b]$}.
	\end{equation}
\end{lemma}
\begin{proof}
	Let $\bar\gamma$ be a curve in $\mathcal{C}([a,b],p,q,\oOmega)$.
	Then $\bar\gamma$ is in $\mathcal{C}_\d$ for any $\delta \in (0,\d_0)$, being $\oOmega \subset \Omega_\d$. 
	Since $\Phi(\bar\gamma(s)) \le 0$ for any $s \in [a,b]$, 
	we set 
	\begin{equation}
	k_1 = \J_\d(\bar\gamma) = 
	\frac{1}{2}\int_{a}^{b}	G(\bar\gamma,\dot{\bar\gamma}) ds 
	+ \int_{a}^{b} \chi_\d(\Phi(\bar{\gamma})) ds 
	= 
	\frac{1}{2}
	\int_{a}^{b}	G(\bar\gamma,\dot{\bar\gamma}) ds < + \infty,
	\end{equation}
	and \eqref{eqn:Jdlimitato} follows.
	Since $\gamma_{\d}$ is a minimum point of $\J_\d$, it is a critical point of $\J_\d$ and by Lemma \ref{lem:regCritJd} there exists a constant $E_\d$ such that \eqref{eqn:conEnerDelta} holds.
	Using $\eqref{eqn:Jdlimitato}$ we get
	\begin{multline}
	\int_{a}^{b}E_\d\ ds = 
	\int_a^b
	\left( 
	\frac{1}{2}G(\gamma_\d,\dot\gamma_\d) - 
	\chi_\d(\Phi(\gamma_\d))
	\right) ds \\ = 
	\J_\d(\gamma_\d) - 2\int_{a}^{b} \chi_\d(\Phi(\gamma_\d))\ ds
	\le k_1 - 2 \int_{a}^{b}\chi_\d(\Phi(\gamma_\d))\ ds \le k_1.
	\end{multline}
	Thus
	\begin{equation}
	E_\d \le \frac{k_1}{b-a}, \quad \text{for all } \d \in (0,\d_0).
	\end{equation}
	and \eqref{eqn:Gdbounded} holds.
\end{proof}
%
\begin{lemma}
	\label{lem:lambdaLinf}
	Let $(\gamma_{\d})_{\d \in (0,\d_0)}$ be a family in $H^{1,2}([a,b],\M)$ such that, for any $\d \in (0,\d_0)$, $\gamma_{\d}$ is a minimum of $\J_{\d}$ on $\mathcal{C}([a,b],p,q,\Omega_\d)$.
	For any $\d \in (0,\d_0)$, set
	\begin{equation}
	\label{eqn:lambdadDef}
	\lambda_\d(s) = - \chi_\d'(\Phi(\gamma_\d(s))), \quad s \in [a,b].
	\end{equation}
	Then there exists a $\d_1 \in (0,\d_0)$ such that 
	\begin{equation}
	\label{eqn:lambdaBound}
	\sup_{\d \in (0,\d_1)}
	\norm{\lambda_\d}_\infty = \sup_{\d \in (0,\d_1)} \max_{s \in [a,b]} |\lambda_\d(s)| < +\infty,
	\end{equation}
\end{lemma}

\begin{proof}
	For any $\d \in (0,\d_0)$, set $\rho_\d(s) = \Phi(\gamma_\d(s))$ and let $s_\d$ be a maximum point for 
	$\rho_\d$.
	Since the derivative of $\chi_\d$ is non-decreasing and $\chi'_\d(t) = 0$ for any $t \le 0$, then
	$$
	0 
	\le 	\chi'_\d(\Phi(\gamma_\d(s)))
	\le \chi'_\d(\Phi(\gamma_\d(s_\d))), \quad \forall s \in [0,1].
	$$
	Thus, it suffices to prove \eqref{eqn:lambdaBound} assuming that
	$$
	\Phi(\gamma_\d(s_\d)) \in ]0,\delta[.
	$$
	Due to the lack of regularity of $G$ on the zero section, we must distinguish the two cases in which $\dot\gamma_\d(s_\d)$ is equal to zero or not.
	In both cases, we will prove the existence of a constant $L > 0$ such that
	\begin{equation}
	\label{eqn:lambdaBoundProof1}
	\chi'_\d(\Phi(\gamma(s_\d)))
	\le
	L \bigg( 1 + \chi_\d(\Phi(\gamma(s_\d)))\bigg),
	\quad
	\text{for any $\d \in (0,\d_0)$,}
	\end{equation}
	from which we infer the thesis.
	Indeed, by \eqref{eqn:chiDeriv} we obtain
	\begin{equation}
	\left(
	\frac{2\d}{\Phi(\gamma_\d(s_\d)) (\d - \Phi(\gamma_\d(s_\d)))}
	- L
	\right) 
	\chi_\d(\Phi(\gamma(s_\d)))
	\le L.
	\end{equation} 
	Since 
	$$
	\inf_{t \in (0,\d)}\frac{2\d}{(t(\d - t))}=  \frac{8}{\d},
	$$
	then setting $\d_1 = 8/L$ we obtain \eqref{eqn:lambdaBound}.
	
	\noindent
	\textit{First case}: Let $s_\d \in A_\d = \{s \in [a,b]: \dot\gamma_\d(s) \ne 0\}$.
	By Lemma \ref{lem:regCritJd}, $\gamma_\d$ is twice differentiable in a neighbourhood of $s_\d$, and being a maximum point of $\rho_\d$ we get that
	$\dot\rho_\d(s_\d) = 0$ and 
	\begin{equation}
	\ddot\rho_\d(s_\d) = 
	{\p^2}_{ q^i q^j}\Phi(\gamma_\d(s_\d))
	\dot\gamma_\d^i \dot\gamma_\d^j
	+
	\p_{q^i} \Phi(\gamma_\d(s_\d))\ddot\gamma_\d^i(s_\d)
	\le 0.
	\end{equation}
	By \eqref{eqn:eqGeoPen} we obtain
	\begin{multline}
	{\p^2}_{ q^i q^j}\Phi(\gamma_\d)
	\dot\gamma_\d^i \dot\gamma_\d^j
	-
	\p_{q^i} \Phi(\gamma_\d)
	\Gamma_{jk}^i(\gamma_\d, \dot\gamma_\d)\dot\gamma_\d^j \dot\gamma_\d^k
	\\
	+
	\chi'_\d(\gamma_\d)
	\p_{q^i} \Phi(\gamma_\d)
	\p_{q^j} \Phi(\gamma_\d)
	g^{ij}(\gamma_\d,\dot\gamma_\d)
	\le 0, 
	\end{multline}
	having omitted the dependency on $s_\d$ for the sake of presentation.
	Since $\overline\Omega_{\d_0}$ is a compact domain and the components of the Chern connection $\Gamma_{jk}^i$ are fiberwise positively homogeneous of degree zero, there exists a constant $c_1 > 0$ such that
	\begin{equation}
	- c_1 G(\gamma_\d,\dot\gamma_\d) 
	\le
	{\p^2}_{ q^i q^j}\Phi(\gamma_\d)
	\dot\gamma_\d^i \dot\gamma_\d^j
	-
	\p_{q^i} \Phi(\gamma_\d)
	\Gamma_{jk}^i(\gamma_\d, \dot\gamma_\d)\dot\gamma_\d^i \dot\gamma_\d^j.
	\end{equation}	
	Moreover, since $d\Phi(q) \ne 0$ for any $q \in \Phi^{-1}([0,\d_0])$ and the matrix $[g^{ij}(q,v)]$ is positive definite for any $v \ne 0$, there exists a constant $c_2 > 0$ such that
	\begin{equation}
	c_2 
	<
	\p_{q^i} \Phi(\gamma_\d)
	\p_{q^j} \Phi(\gamma_\d)
	g^{ij}(\gamma_\d,\dot\gamma_\d).
	\end{equation}
	Hence we obtain 
	\begin{equation}
	- c_1 G(\gamma_\d, \dot\gamma_\d)
	+
	c_2
	\chi'_\d(\gamma_\d)
	\le 0,
	\end{equation}
	and by \eqref{eqn:Gdbounded} we get the existence of a constant $L >0 $ such that \eqref{eqn:lambdaBoundProof1} holds.
	
	\noindent
	\textit{Second case}:
	Let $s_\d \in B_\d = [a,b]\setminus A_\d$.
	Firstly, we show that the interior of $B_\d$ is empty.
	Arguing by contradiction, let $I(s_\d)$ a neighbourhood of $s_\d$ such that $I(s_\d) \subset B_\d$.
	Take a vector field $\xi \in T_{\gamma_\d}\M$, with support in $I_\d$, such that $d\Phi(\gamma_{\d}(s_\d))[\xi(s_\d)] < 0$
	and $d\Phi(\gamma_{\d}(s))[\xi(s)] \le 0$ for any $s \in I(s_\d)$.
	Since $\dot\gamma_\d(s) = 0$ on $I(s_\d)$,
	then
	$$
	d_qG(\gamma_\d(s),\dot\gamma_\d(s)) = d_vG(\gamma_\d(s),\dot\gamma_\d(s)) = 0,
	\quad \text{for any $ s \in I(s_\d)$}.
	$$
	Moreover, we can choose the neighbourhood $I_\d$ such that $\Phi(\gamma_{\d}(s)) > 0$ for any $s \in I_\d$, so $\chi'_\d(\gamma_{\d}) > 0$ in $I_\d$.
	As a consequence, being $\gamma_\d$ a critical point for $\J_\d$, we get
	\begin{multline}
	0 = d\J_\d(\gamma_\d)[\xi] = 
	\frac{1}{2}\int_{I(s_\d)}
	\left(
	d_qG(\gamma_\d,\dot\gamma_\d)[\xi]
	+
	d_vG(\gamma_\d,\dot\gamma_\d)[\dot\xi]
	\right)ds \\
	+  \int_{I(s_\d)}
	\chi'_\d(\gamma_\d)d\Phi(\gamma_\d)[\xi]ds
	=
	\int_{I(s_\d)}
	\chi'_\d(\gamma_\d)d\Phi(\gamma_\d)[\xi]ds < 0,
	\end{multline}
	which is an absurd.
	Now, since the interior of $B_\d$ is empty, we must distinguish two cases.
	\begin{itemize}
		\item[(i)] $s_\d$ is an isolated point of $B_\d$.
		Then, recalling that $\rho_\d \in C^1$ and that $s_\d$ is a maximum point for $\rho_\d$, $\dot\rho_\d(s_\d) = 0$. Moreover, there exists a neighbourhood $I_\d$ of $s_\d$ such that 
		$I_\d \cap B_\d = \{s_\d \}$ and, for any $s \in I_\d$, $\dot\rho_\d(s) \ge 0$ for $s < s_\d$.
		Applying the mean value theorem, we can construct a sequence $(s_{\d,n})_n$ that converges to $s_\d$ from below and such that $\ddot\rho_\d(s_{\d,n}) \le 0$.
		Following the same procedure of the first case, we obtain that 
		\begin{equation}
		\chi'_\d(\Phi(\gamma(s_{\d,n})))
		\le
		L \bigg( 1 + \chi_\d(\Phi(\gamma(s_{\d,n})))\bigg)
		,
		\end{equation}
		and passing to the limit we obtain \eqref{eqn:lambdaBoundProof1} for $s_\d$.
		
		\item[(ii)] $s_\d$ is an accumulation point of $B_\d$.
		In this case we can suppose, without loss of generality, that there exists a strictly increasing sequence $(s_{\d,n})_n$ in $B_\d$ that converges to $s_\d$ and such that $(s_{\d,n-1},s_{\d,n}) \subset A_\d$.
		Since $\dot\gamma_\d(s_{\d,n}) = 0$, $\dot\rho_\d(s_{\d,n}) = 0$ for any $n$ and we can apply the Rolle's theorem to construct a sequence $(\bar{s}_{\d,n})_n$ in $A_\d$ such that
		$\bar{s}_{\d,n} \in (s_{\d,n-1},s_{\d,n})$ and $\ddot\rho(\bar{s}_{\d,n}) = 0$.
		Then applying the same reasoning of the first case and passing to the limit as $\bar{s}_{\d,n} \to s_\d$, we obtain
		\eqref{eqn:lambdaBoundProof1} for $s_\d$. 
	\end{itemize}
\end{proof}


\begin{lemma}
	\label{lem:lim-delta}
	Let $(\gamma_{\d})_{\d \in (0,\d_0)}$ be a family in $H^{1,2}([a,b],\M)$ such that, for any $\d \in (0,\d_0)$, $\gamma_{\d}$ is a minimum of $\J_{\d}$ on $\mathcal{C}([a,b],p,q,\Omega_\d)$.
	Then there exists a subsequence $(\d_n)_n$ in $(0,\d_0)$ such that
	\begin{enumerate}
		\item $(\gamma_{\d_n})_n$ strongly converges to a curve $\gamma \in \mathcal{C}([a,b],p,q,\overline{\Omega})$;
		\item the sequence of functions $(\lambda_{\d_n})_n$
		weakly converges to a function $\lambda \in L^2([a,b],\R)$;
		\item if $p \ne q$, the limit curve $\gamma$ is such that $\dot{\gamma}(s) \ne 0$ for all $s \in [a,b]$, it 
		has $H^{2,2}$-regularity on $[a,b]$ and it satisfies a.e.
		\begin{equation}
		\label{eqn:ELboundarydelta}
		d_qG(\gamma, \dot{\gamma}_d) - \frac{d}{ds}d_vG(\gamma_\d, \dot{\gamma}_d) =  \lambda\ \nabla\Phi(\gamma).
		\end{equation}
		\item[(4)] the limit curve $\gamma$ is a minimum of $\J^{a,b}$ on $\mathcal{C}([a,b],p,q,\oOmega)$.
	\end{enumerate}
\end{lemma}
\begin{proof}
	For any $\d \in (0,\d_0)$, $\gamma_{\d}$ is a minimum of $\J_\d$, so \eqref{eqn:Jdlimitato} holds and 
	$$
	\int_{a}^{b}G(\gamma_\d,\dot\gamma_\d)ds \le k_1, \qquad \forall \d \in (0,1).
	$$
	By \eqref{eqn:boundG} and the Ascoli-Arzel\'a theorem, we obtain a decreasing sequence $(\d_n)_n \subset (0,1)$ that converges to $0$ such that $\gamma_{\d_n}$ uniformly converges to a curve $\gamma$ and $\dot\gamma_{\d_n}$ weakly converges to $\dot{\gamma}$.
	Since $\gamma_{\d_n}(s) \subset \Omega_{\d_n}$ for every $n \in \N$ and $\d_n \to 0$, the support of $\gamma$ is in $\overline{\Omega}$.
	To prove the statement (1), it remains to show that $\dot\gamma_{\d_n}$ strongly converges to $\dot{\gamma}$.
	Consider a smooth curve $\omega \in \mathcal{C}([a,b],p,q,\M)$ that approximates $\gamma$
	and set $\xi_n(s) = \text{exp}^{-1}_{\omega(s)}(\gamma_{\d_n}(s))$, where $\text{exp}$ is the exponential map of the Riemannian metric that derives from the euclidean product.
	For every $n,m \in \N$ we set
	$$
	V_{n,m} = d\text{exp}_\omega[\xi_m - \xi_n].
	$$		
	As $(\gamma_{\d_n})_n$ uniformly converges to $0$, also $V_{n,m}$ uniformly converges to $0$ as $n,m \to + \infty$.
	Since $\gamma_{\d_n}$ is a critical point for $\J_{\d_n}$, we obtain
	\begin{equation}
	\label{eqn:dtozeroProof1}
	d\J_{\d_n}(\gamma_{\d_n})[V_{n,m}] 
	=
	d\J(\gamma_{\d_n})[V_{n,m}] 
	+
	\int_{a}^{b}
	\lambda_{\d_n}
	\sprod{\nabla\Phi(\gamma_{\d_n})}{V_{n,m}} \ ds
	= 0.
	\end{equation}
	By Lemma \ref{lem:lambdaLinf}, if $n$ is sufficiently large, then $\lambda_{\d_n}$ is bounded in $L^{\infty}$. 
	Since $V_{n,m}$ uniformly converges to $0$, this implies that the right-hand side of \eqref{eqn:dtozeroProof1} goes to zero, and we obtain
	\begin{equation}
	\label{eqn:dtozeroProof2}
	\lim_{m,n \to \infty}
	d\J(\gamma_{\d_n})[V_{n,m}]  = 0.
	\end{equation}
	From this last equality, we can obtain the strong convergence of $\dot\gamma_{\d_n}$ to $\dot{\gamma}$, up to subsequences,
	following the same procedure exploited in the proof of \cite[Theorem 3.1]{CaponioMasiello2011}.
	
	Statement (2) naturally derives from Lemma \ref{lem:lambdaLinf}.
	Indeed, from the sequence $(\d_n)_n$ obtained in the first statement we can select a subsequence, which we denote again by $(\d_n)_n$, such that $\lambda_{\d_n}$ weakly converges in $L^2([a,b])$.
	
	In order to prove the statement (3), firstly we use \cite[Proposition 4.6]{Bartolo2011} to obtain that $\gamma$
	is $C^1$ and, for any $\bar{s} \in [a,b]$ such that $\dot{\gamma}_\d(\bar{s}) \ne 0$, said $(U,\varphi)$ a chart of $\M$ such that $\gamma_\d(\bar{s})\in U$, $\gamma$
	has $H^{2,2}$-regularity on an open subset of $\gamma^{-1}(U)$ containing the point $\bar{s}$ and there it satisfies \eqref{eqn:ELboundarydelta} a.e..
	Now set $A = \{s \in [a,b]: \dot{\gamma}(s) \ne 0\}$.
	Since we are assuming that $p \ne q$, the set $A$ is not empty.
	Let $s_0 \in A$ be such that $\gamma_\d(s_0) \in \Omega$.
	Then there exists a neighbourhood $I$ of $s_0$ and a small number $d > 0$ such that
	$| \Phi(\gamma_{\d_n}(s)) | > d $ for every $s \in I$ and for $n$ sufficiently large.
	As a consequence, $(\lambda_{\d_n})_n$ uniformly converges to $0$ in $I$ and $\lambda = 0$ a.e. in $I$.
	Now set $B = \{ s \in A: \gamma(s) \in \p\Omega\}$.
	Since $\gamma$ is a $C^1$ function and $\Phi(\gamma) = 0$ in $B$, then (cf. \cite[Lemma 7.7]{Gilbarg1983})
	\begin{equation}
	\sprod{\nabla\Phi(\gamma(s))}{\dot{\gamma}(s)} = 0, \quad \forall s\in B.
	\end{equation}
	Contracting both therms of \eqref{eqn:ELboundarydelta} with $\dot{\gamma}$ leads to 
	\begin{equation}
	\left( d_qG(\gamma, \dot{\gamma}) - \frac{d}{ds}d_vG(\gamma, \dot{\gamma}) \right)[\dot{\gamma}] = 
	\lambda \sprod{\nabla\Phi_\d(\gamma(s))}{\dot{\gamma}(s)} = 0.
	\end{equation}
	As a consequence, 
	\begin{equation}
	E(\gamma, \dot{\gamma}) = d_vG(\gamma,\dot{\gamma}) - G(\gamma,\dot{\gamma})
	\end{equation}
	is constant on every connected component of $A$.
	By Euler's theorem, $E(\gamma, \dot{\gamma})  = G(\gamma,\dot{\gamma})$ and, being $ s \in [a,b] \to G(\gamma(s),\dot{\gamma}(s))$ a continuous function, we conclude that $G(\gamma,\dot{\gamma})$ is a non zero constant on the whole $[a,b]$.
	Consequently, $A = [a,b]$ and \eqref{eqn:ELboundarydelta} holds on the whole $[a,b]$. 
	
	Let us prove statement (4).
	Recalling \eqref{eqn:chiDeriv}, by Lemma \ref{lem:lambdaLinf}
	there exist a constant $k > 0$ such that
	\begin{multline}
	\sup_{s \in [a,b]}
	\left|\chi_{\d_n}(\Phi(\gamma_{\d_n}(s)))\right|
	= 
	\sup_{s \in [a,b]}
	\left|
	\chi'_{\d_n}(\Phi(\gamma_{\d_n}(s)))
	\frac{\Phi(\gamma_{\d_n}(s))
		\left(\d_n - \Phi(\gamma_{\d_n}(s))	\right)
	}{2\d_n}
	\right| \\
	\le
	k \sup_{t \in (0,\d_n)}
	\left|
	\frac{t (\d_n - t)}{2\d_n}
	\right|
	= \frac{k}{8} \d_n \to 0.
	\end{multline}
	Consequently, 
	$$
	\lim_{n \to \infty} \int_{a}^{b} \chi_{\d_n}(\Phi(\gamma_{\d_n}(s))) \ ds = 0,
	$$
	and since $\mathcal{C}([a,b],p,q,\oOmega) \subset \mathcal{C}([a,b],p,q,\Omega_\d)$ for any $\d$, then 
	$$
	\J^{a,b}(\gamma) = 
	\lim_{n \to \infty} 
	\J_{\d_n}(\gamma_{\d_n})
	\le 
	\lim_{n \to \infty}  
	\J_{\d_n}(y)
	=
	\J^{a,b}(y),
	\qquad \text{for all } y \in \mathcal{C}([a,b],p,q,\oOmega).
	$$
\end{proof}


\begin{proof}[Proof of Proposition \ref{prop:regularity}]
	If $x$ is a constant curve, i.e. $x(s) = q \in \p\Omega$ for any $s \in [0,1]$, then it is obviously in $H^{2,\infty}$ and \eqref{eqn:ELboundary} is trivially satisfied setting $\lambda = 0$.
	Let us assume that $x$ is a non-constant curve.
	We need only to prove the regularity of $x$ when it touches the boundary $\p\Omega$. In fact,
	$x$ is a free geodesic when it lays on $\Omega$, 
	so it is $C^2$.
	Since the regularity is a local property, we can restrict our analysis on a single chart $(U,\varphi)$ in a neighbourhood of a point $x(\bar{t}) \in \p\Omega$.
	Let $(a,b)$ be a neighbourhood of $\bar{t}$ such that $x([a,b]) \subset U$ and $x(a) \ne x(b)$. If $\bar{t} = 0$, then set $a = 0$ and, similarly, if $\bar{t} =1$, then set $b = 1$.
	Notice that, for our purpose, we can choose $a$ and $b$ as close as we desire.
	
	For any $\delta \in (0,\d_0)$, consider the functional $\J_{\d}$ defined on $\mathcal{C}([a,b],x(a),x(b),\Omega_\d)$.
	By Lemma \ref{lem:lim-delta}, there exists a curve $\gamma \in H^{2,2}$ that is a minimum of $\J^{a,b}$ on $\mathcal{C}([a,b],x(a),x(b),\overline{\Omega})$.
	As a first step, we prove that $x$ has $H^{2,2}$ regularity by showing that $x = \gamma$.
	
	Let us show that if $a$ and $b$ are sufficiently close, then $\gamma([a,b]) \subset U$.
	Let $y$ be a curve $\mathcal{C}([a,b],x(a),x(b),\overline{\Omega})$
	such that $y([a,b]) \not\subset U$ and let $\bar{s}$ be the first instant at which $y(s) \not\in U$.
	Using the Cauchy-Schwarz inequality and \eqref{eqn:boundG} we obtain that
	\begin{equation}
	\dist (x(a),\p U) \le 
	\int_{a}^{\bar{s}} \norm{\dot{y}(s)} ds
	\le 
	\left(
	\ell(b-a)
	\J(y)
	\right)^{\frac{1}{2}},
	\end{equation}
	so
	\begin{equation}
	\label{eqn:regProof1}
	\J(y) \ge  \frac{\dist^2 (x(a),\p U)}{\ell(b-a)}.
	\end{equation}
	Now let $\tilde x$ be 
	the reparametrization of $x$ 
	such that $\tilde{x}([a,b]) = x([a,b])$ and $G(\tilde x(s),\dot{\tilde{x}}) = c_x \ne 0$.
	Then 
	\begin{equation}
	\label{eqn:regProof2}
	\J^{a,b}(\gamma) = 
	\int_{a}^{b}G(\gamma,\dot{\gamma})\ ds
	\le 
	\int_{a}^{b}G(\tilde{x},\dot{\tilde{x}})\ ds 
	= (b - a) c_x.
	\end{equation}
	As a consequence, choosing $a$ and $b$ such that
	\begin{equation}
	b - a <
	\frac{\dist (x(a),\p U)}{ \sqrt{\ell c_x}} \ 
	,
	\end{equation}
	from \eqref{eqn:regProof1} and \eqref{eqn:regProof2} we obtain that
	\begin{equation}
	\J^{a,b}(\gamma)  <\frac{\dist^2 (x(a),\p U)}{ \ell (b-a)},
	\end{equation}
	so it lays on $U$.
	
	Now choose the map $\varphi$ such that 
	$$
	d\varphi\left(\frac{\nabla\Phi(q)}{\norm{\nabla\Phi(q)}}\right) \in \R^N
	$$
	is constant on the chart.
	Then $\xi = \gamma - x$ is an admissible variation of $x$ in $\mathcal{C}([a,b],x(a),x(b),\overline{\Omega})$,
	since 
	$$
	\sprod{\xi(s)}{\nabla\Phi(x(s))} \le 0, \qquad 
	\text{if } x(s) \in \p\Omega.
	$$
	Now define $f:[0,1] \to \R$ by
	\begin{equation}
	\label{eqn:regProof3-fDef}
	f(t) = \J^{a,b}(x + t\xi).
	\end{equation}
	Since $\gamma$ is a minimum for $\J^{a,b}$, we have that
	$$
	f(1) - f(0)
	=\J^{a,b}(\gamma) -  \J^{a,b}(x) \le 0.
	$$
	Setting $\xi(s) = 0$ for any $s \in [0,1] \setminus [a,b]$, we have that $\xi \in \V^-(x,\Mf)$.
	Since $x$ is a $\V^-$-critical curve for $\J$ on $\Mf$, we obtain 
	\begin{equation}
	\label{eqn:regProof4}
	f'(0) 
	= d\J^{a,b}(x)[\xi] 
	= d\J(x)[\xi] \ge 0.
	\end{equation}
	Looking for a contradiction, we set $\gamma \ne x$ and show that if $a$ and $b$ are sufficiently close then
	\begin{equation}
	\label{eqn:qwer}
	\int_{0}^{1}\big( f'(t) - f'(0)\big) dt > 0.
	\end{equation}
	As a consequence, 	\begin{equation}
	0  \ge f(1) - f(0) 
	= f'(0) + \int_0^1 \big( f'(t) - f'(0)\big) dt 
	\ge  \int_0^1 \big( f'(t) - f'(0)\big) dt > 0,
	\end{equation}
	which is an absurd.
	By definition of $f$ we have 
	\begin{multline}
	f'(t) - f'(0) = 
	\frac{1}{2} \int_{a}^{b}
	\bigg[
	\left(
	d_qG(x + t \xi,\dot{x} + t \dot\xi)[\xi] 
	- d_qG(x ,\dot{x} )
	\right)[\xi] \\
	+ \left(
	d_vG(x + t \xi,\dot{x} + t \dot\xi) - 
	d_vG(x ,\dot{x} )
	\right)
	[\dot{\xi}]
	\bigg] 
	ds.
	\end{multline}
	Using the mean value theorem and the bounds in \eqref{eqn:boundDeriv2} and \eqref{eqn:defAlpha} we obtain
	\begin{multline}
	\label{eqn:asdf}
	f'(t) - f'(0) = 
	\frac{t}{2}
	\int_{a}^{b}
	ds
	\int_0^1
	\bigg(
	d_{vv}G(x + \sigma t \xi,\dot{x} + \sigma t \dot\xi)[ \dot\xi][\dot\xi]\\
	+ 2 d_{qv}G(x + \sigma t \xi,\dot{x} + \sigma t \dot\xi)[ \xi][\dot\xi] 
	+ 	d_{qq}G(x + \sigma t \xi,\dot{x} + \sigma t \dot\xi)[ \xi][\xi] 
	\bigg) 
	d\sigma \\
	\ge
	\frac{t}{2}
	\int_{a}^{b}
	ds
	\int_0^1
	\bigg(
	\alpha \norm{\dot\xi}^2
	- 2 \ell (1 + \norm{\dot{x} + \sigma t \dot\xi})\norm{\xi}\norm{\dot\xi} \\
	- \ell (1 + \norm{\dot{x} + \sigma t \dot\xi}^2)\norm{\xi}^2
	\bigg) 
	d\sigma.
	\end{multline}	
	Let us show that there exists a constant $c_1 > 0$, which depends only on $x$, such that
	\begin{equation}
	\int_{a}^{b} 
	\norm{\dot{x} + \sigma t \dot\xi}^2
	ds
	\le c_1.
	\end{equation}
	Indeed, by the inequalities in \eqref{eqn:boundG} and since $\gamma$ is a minimum for $\J^{a,b}$, we have the following chain of inequalities
	\begin{multline}
	\int_{a}^{b} 
	\norm{\dot{x} + \sigma t \dot\xi}^2
	ds
	\le
	2 \int_{a}^{b}
	\left(
	\norm{\dot{x}}^2 + \norm{\sigma t \dot\xi}^2
	\right)	ds
	\le
	2 \int_{a}^{b}
	\left(
	\norm{\dot{x}}^2 + \norm{\dot\xi}^2
	\right)	ds \\
	\le
	2 \int_{a}^{b}
	\left(
	\norm{\dot{x}}^2 + 2 (\norm{\gamma}^2 + \norm{x}^2)
	\right) ds 
	\le
	6 \int_{a}^{b} \norm{\dot{x}}^2 ds
	+
	4 \ell \int_{a}^{b} G(\gamma,\dot{\gamma}) ds
	\\
	\le
	6 \int_{a}^{b} \norm{\dot{x}}^2 ds
	+
	4 \ell \int_{a}^{b} G(x,\dot{x}) ds
	\le
	6 \int_{a}^{b} \norm{\dot{x}}^2 ds
	+
	4 \ell^2 \int_{a}^{b} \norm{\dot{x}}^2 ds
	\\
	\le
	(6 + 4 \ell^2)\int_{0}^{1} \norm{\dot{x}}^2 ds = c_1.
	\end{multline} 
	As a consequence, there exists a strictly positive constant $c_2$ such that
	\begin{multline}
	I_1 = 
	\ell
	\int_{a}^{b} \left(
	\int_{0}^{1}
	( 1+ \norm{\dot{x} + \sigma t \dot\xi})\norm{\xi}\norm{\dot\xi} d \sigma
	\right) ds \\
	=
	\ell
	\int_{a}^{b}
	\norm{\xi}\norm{\dot\xi} ds
	+
	\ell
	\int_{0}^{1}
	\left(
	\int_{a}^{b}
	(\norm{\dot{x} + \sigma t \dot\xi})
	\norm{\xi}\norm{\dot\xi}
	ds
	\right) d\sigma 
	\\
	\le
	\ell
	\norm{\xi}_{L^\infty} \norm{\dot\xi}_{L^2}
	+
	\ell
	\norm{\xi}_{L^\infty} \norm{\dot\xi}_{L^2}
	\int_{0}^{1}
	\left(\int_{a}^{b} \norm{\dot{x} + \sigma t \dot\xi}^2 ds \right)^{1/2}
	d\sigma
	\\
	\le 
	c_2 \norm{\xi}_{L^\infty} \norm{\dot\xi}_{L^2},
	\end{multline}	
	where we applied the Tonelli's theorem and the H\"{o}lder inequality. 
	Similarly, there exists a constant $c_3 > 0$ such that
	\begin{multline}
	\label{eqn:I1bound}
	I_2 =
	\ell
	\int_{a}^{b} \left(
	\int_{0}^{1}
	(1 + \norm{\dot{x} + \sigma t \dot\xi}^2)\norm{\xi}^2 d\sigma
	\right) ds \\
	\le
	\ell
	\norm{\xi}^2_{L^\infty}
	\int_{0}^{1}
	\left(
	\int_{a}^{b} (1 + \norm{\dot{x} + \sigma t \dot\xi}^2)
	ds
	\right) d\sigma
	\le 
	c_3
	\norm{\xi}^2_{L^\infty}.
	\end{multline}
	Then, by \eqref{eqn:asdf} we obtain 
	\begin{equation}
	\label{eqn:asdf2}
	f'(t) - f'(0)
	\ge
	\frac{t}{2}
	\left(
	\alpha\norm{\dot\xi}^2_{L^2}
	- c_2 \norm{\xi}_{L^\infty} \norm{\dot\xi}_{L^2}
	- c_3 \norm{\xi}^2_{L^\infty}
	\right).
	\end{equation}
	Since $\xi(a) = 0$, we have
	\begin{equation}
	\norm{\xi(s)} = 
	\left\lVert
	\xi(a) + \int_{a}^{s} \dot\xi(\sigma)d\sigma 
	\right\rVert
	\le
	\int_{a}^{s} \norm{\dot\xi(\sigma)} d\sigma \le 
	\sqrt{s - a}\ \norm{\dot\xi}_{L^2},
	\end{equation}
	therefore
	\begin{equation}
	\label{eqn:xiInfledotxi2}
	\norm{\xi}_{L^\infty} \le \sqrt{b - a}\ \norm{\dot\xi}_{L^2}.
	\end{equation}
	By \eqref{eqn:asdf2} we obtain
	\begin{equation}
	f'(t) - f'(0) \ge
	\frac{t}{2} \norm{\dot\xi}^2_{L^2}
	\left(\alpha -c_2\sqrt{b - a} - c_3(b- a)\right),
	\end{equation}
	and, if $b-a$ is sufficiently small, then there exists a constant $c_4 > 0$ such that
	$$
	f'(t) - f'(0) \ge c_4 t. 
	$$
	As a consequence, \eqref{eqn:qwer} holds so $x$ and $\gamma$ must coincide.
	
	Since $x = \gamma$ in every chart, by Lemma \ref{lem:lim-delta} we obtain that
	$\dot{x}(s) \ne 0$ for all $x \in [0,1]$.
	Moreover, \eqref{eqn:ELboundary} holds a.e. for a function $\lambda \in L^2([0,1],\R)$ such that $\lambda(s) = 0$ for all $s \notin C_x$ and $\lambda \le 0$ a.e..
	Set $\rho(s) = \Phi(x(s))$.
	Since $\rho(s) = 0$ on $C_x$ and $\dot\rho$ is absolutely continuous, by \cite[Lemma 7.7]{Gilbarg1983} we have
	$
	\ddot\rho(s) = 0 \text{ a.e. on $C_x$}.
	$
	Using \eqref{eqn:ELboundary} in local coordinates, we get that $x$ satisfies the equations
	\begin{equation}
	\label{eqn:eqGeoBound}
	\ddot x^i + \Gamma_{jk}^i(x,\dot{x})
	\dot{x}^j \dot{x}^k =
	-\lambda \p_{q^k}\Phi({x})
	g^{ki}({x},\dot{x}),
	\end{equation}
	so we obtain
	\begin{equation}
	0 = 
	\ddot{\rho}(s) = 
	\p^2_{q^i q^j}\Phi(x)\dot{x}^i\dot{x}^j -
	\p_{q^i} \Phi(x)\Gamma_{jk}^i(x,\dot{x})
	\dot{x}^j \dot{x}^k
	- 
	\lambda \p_{q^i}\Phi(x)\p_{q^k}\Phi(x)g^{ik}(x,\dot{x}).
	\end{equation}
	Using \eqref{eqn:FinHesLocCoo}, we obtain \eqref{eqn:lambdaCx}, and since $\dot{x}$ is a continuous function, also $\lambda$ is a continuous function in $[0,1]$. Then $\lambda \in L^\infty([0,1],\R)$ and using  \eqref{eqn:eqGeoBound} we obtain that $\ddot{x} \in L^\infty([0,1],\R^N)$, so $x \in H^{2,\infty}([0,1],\oOmega)$.
\end{proof}

\begin{remark}
	By a similar argument used in the proof of Proposition \ref{prop:regularity}, one can prove that if $q_1,q_2 \in \oOmega$ are sufficiently close, then there exists a unique $\V^-$-critical curve for $\J$ on $\Mf$, i.e. a unique geodesic in a manifold with boundary, that connects $q_1$ and $q_2$.
	This has been done for the Riemannian case in \cite{Scolozzi1986}.
\end{remark}

Thanks to the $H^{2,\infty}$-regularity of the $\V^-$-critical curves for $\J$ on $\Mf$ given by Proposition \ref{prop:regularity},  integration by parts of \eqref{eqn:Jdiff} becomes available.
This allows to prove that every non-constant $\V^-$-critical curve for $\J$ on $\Mf$ either it is an orthogonal Finsler geodesic chord or it contains an orthogonal-tangent Finsler geodesic chord.
%

\begin{lemma}
	\label{lem:orthogonality}
	If $x$ is a non-constant $\V^-$-critical curve of $\J$ in $\Mf$, then $\dot{x}(0)$ and $\dot{x}(1)$ are orthogonal to $\p\Omega$.
\end{lemma}
\begin{proof}
	Set $C_x = \{s \in [0,1]: x(s) \in \p\Omega \}$ and $I_x = [0,1]\setminus C_x$.
	Let $\xi \in T_x\M$ be such that 
	$$
	\sprod{\xi(s)}{\nabla\Phi(x(s))} = 0  \qquad \text{on $C_x$}.
	$$
	Then, both $\xi$ and $-\xi$ are admissible infinitesimal variations of $x$ in $\Mf$ and
	$
	d\J(x)[\xi] = 0.
	$
	As a consequence, partial integration and \eqref{eqn:ELboundary} lead to 
	\begin{multline}
	d\J(x)[\xi] = 
	\frac{1}{2}\bigg(
	d_vG(x(1),\dot{x}(1))[\xi(1)] - 
	d_vG(x(0),\dot{x}(0))[\xi(0)] 
	\bigg)\\
	+ \frac{1}{2}
	\int_{0}^1 \left(
	d_q G(x,\dot{x}) -\frac{d}{ds} d_vG(x,\dot{x})
	\right) 
	[\xi]
	ds
	\\
	=
	\frac{1}{2}\bigg(
	d_vG(x(1),\dot{x}(1))[\xi(1)] - 
	d_vG(x(0),\dot{x}(0))[\xi(0)] 
	\bigg)
	= 0.
	\end{multline}
	Since both $\xi(0)$ and $\xi(1)$ are arbitrary tangent vectors to $\p\Omega$, we infer that, with respect to the Finsler structure, both $\dot{x}(0)$ and $\dot{x}(1)$ are orthogonal to $T_{x(0)}\p\Omega$ and $T_{x(1)}\p\Omega$ respectively. 
\end{proof}

\begin{proposition}
	\label{prop:V-crit-OFG}
	Let $x$ be a non-constant $\V^-$-critical curve for $\J$ on $\Mf$.
	Then either
	\begin{itemize}
		\item there exists $\bar{s} \in ]0,1[$ such that $x|_{[0,\bar{s}]}$ is an orthogonal-tangent Finsler geodesic chord,
	\end{itemize}
	or
	\begin{itemize}
		\item $x$ is an orthogonal Finsler geodesic chord.
	\end{itemize}
\end{proposition}
\begin{proof}
	Set $C_x = \{ s \in ]0,1[: x(s) \in \p\Omega \}$.
	Since $x$ is a non-constant $\V^-$-critical curve for $\J$ on $\Mf$, $\dot{x} \ne 0$ everywhere.
	By Lemma \ref{lem:orthogonality}, $\dot{x}(0)$ points inside $\Omega$.
	As a consequence, if $C_x \neq \emptyset$, then $\bar s = \min C_x > 0$.
	By Proposition \ref{prop:regularity}, $x$ is of class $C^1$ and $\dot{x}(\bar{s})$ must be tangent to $\p\Omega$.
	Then $x|_{[0,\bar{s}]}$ is an orthogonal-tangent Finsler geodesic chord.
	Otherwise, if $C_x = \emptyset$, Lemma \ref{lem:orthogonality} implies that $x$ is an orthogonal Finsler geodesic chord.
\end{proof}

\begin{remark}
	Let $\Omega$ be a strictly convex domain, i.e.
	every Finsler geodesic that is tangent to $\p\Omega$ 
	locally remains outside $\Omega$.
	Then every no-constant $\V^-$-critical curve for $\J$ on $\Mf$ is an \OFGC.
	Indeed, if $\Omega$ is a strictly convex domain, 
	then the Finsler Hessian $H_\Phi(q,v)[v,v] > 0$ for every $q \in \p\Omega$ and $v \in T_q\p\Omega \setminus 0$.
	As a consequence, from \eqref{eqn:lambdaCx} we infer that if $x$ is a non-constant $\V^-$-critical curve for $\J$ on $\Mf$, $\lambda$ should be non-negative whenever $x$ touches the boundary $\p\Omega$. 
	Since $\lambda \le 0$ for every $s \in [0,1]$, then $\lambda = 0$ and $x$ is a free geodesics.
	Therefore, $x(]0,1[) \in \Omega$ and, by Lemma \ref{lem:orthogonality}, $x$ is an \OFGC.
\end{remark}

\section{$\V^-$-Palais-Smale Sequences}

In this section we prove the so-called Palais-Smale condition for the functional $\J$ in $\Mf$.
Due to the presence of the boundary $\p\Omega$, $\Mf$ is not a smooth manifold and the definition of Palais-Smale sequence has to be modified accordingly.

\begin{definition}
	A sequence $(x_n)_n \subset \Mf$ is a $\V^-$-\textit{Palais-Smale sequence} for $\J$ at level $c \in \R$ if
	$
	\lim_{n \to \infty} \J(x_n) = c,
	$
	and for all $V_n \in \V^-(x_n)$ such that $\norm{V_n}_* = 1$
	(where the norm $\norm{\cdot}_*$ has been defined in \eqref{eqn:defNorm*}) 
	the following holds:
	\begin{equation}
	\label{eqn:V-PSdef}
	d\J(x_n)[V_n] \ge - \epsilon_n,
	\end{equation}
	where $\epsilon_n\to 0^+$.
\end{definition}

When the action of a general Tonelli-Lagrangian function $L$ of class $C^2$ is considered, the Palais-Smale condition can be proved exploiting the fiberwise convexity of $L$ by the mean value theorem on $d_vL$ (cf. 
\cite{Asselle2016,Corona2020}).
Since the energy function $G$ is not regular on the zero section, we cannot apply the mean value theorem on $d_vG$.
In \cite{CaponioMasiello2011} a proof of the Palais-Smale condition for the Finsler energy functional is given using a localization argument introduced in \cite{Abbondandolo2007}, which allows to work in an open subset of $\R^n$.
Here we use the same procedure of \cite{CaponioMasiello2011} to deal with the lack of regularity of $G$ on the zero section, but we work in local charts.

\begin{proposition}
	\label{prop:PS}
	Let $(x_n)_n\subset \Mf$ be a $\V^-$-Palais-Smale sequence for $\J$ at level $c \in \R$.
	Then $x_n$ is strongly $H^1$-convergent.
\end{proposition}
\begin{proof}
	If $(x_n)_n$ is a $\V^-$-Palais-Smale sequence, then it admits a subsequence which is $H^1$-weakly convergent.
	In fact, by Remark \ref{rmk:aDef}, for $n$ sufficiently large we have
	$$
	\norm{\dot{x}_n}_2^2 = \int_{0}^1 \norm{\dot{x}_n}^2 ds \le \ell \int_{0}^{1}G(x_n,\dot{x}_n) ds \le 2\ell \J(x_n) \le 2\ell(c+1).
	$$
	Applying the Ascoli-Arzel\'a theorem, there exists a subsequence, which we will denote again by $(x_n)$, that converges uniformly to a curve $x \in \Mf$, so it is weakly $H^1$-convergent.
	Let us consider an auxiliary Riemannian metric $\bar{g}$ in $\overline{\Omega}$ for which the boundary $\p \Omega$ is totally geodesic,
	and let $\overline{\text{exp}}$ denote the relative exponential map.
	For all $n$ sufficiently large, we set
	\begin{equation}
	V_n(s) = (\overline{\text{exp}}_{x_n(s)})^{-1}(x(s)).
	\end{equation}
	Since $x_n$ converges to $x$ uniformly, the vector field $V_n$ is well defined and converges uniformly to $0$.
	Also, $\norm{V_n}_{*}$ is bounded, because $(x_n)_n$ is bounded in $H^1$.
	Moreover, $V_n \in \V^-(x_n)$ for all $n$ sufficiently large.
	In fact, if $x_n(s) \in \Omega$, then $V_n(s)$ can be any element of $T_{x_n(s)}\M$ and if $x_n(s) \in \p\Omega$, we have two cases.
	If $x(s) \in \p\Omega$, then $V_n(s) \in T_{x_n(s)}\p\Omega$, having used the fact that $\p\Omega$ is totally geodesic relatively to $\bar{g}$.
	If $x(s) \in \Omega$, then $V_n(s)$ points inside $\Omega$, since $\overline{\text{exp}}$ locally defines a chart in a neighbourhood of $x_n(s)$.
	As a consequence, being $(x_n)_n$ a $\V^-$-Palais-Smale sequence, we have that
	\begin{equation}
	\liminf_{n \to \infty} d\J(x_n)[V_n] 
	=\liminf_{n \to \infty} \frac{1}{2} \int_{0}^1
	\left( 
	d_qG(x_n,\dot{x}_n)[V_n] +
	d_vG(x_n,\dot{x}_n)[\dot{V}_n] \right)ds
	\ge 0.
	\end{equation}
	By \eqref{eqn:boundDqG} and since $V_n \to 0$ uniformly, 
	\begin{equation}
	\label{eqn:PSproof1}
	\lim_{n \to \infty}\int_{0}^1 d_qG(x_n,\dot{x}_n)[V_n] ds 
	\le \lim_{n \to \infty} \int_{0}^1 c_3 (1+ \norm{x_n}^2)\norm{V_n} ds = 0,
	\end{equation}
	and, as a consequence, 
	\begin{equation}
	\label{eqn:PSproof2}
	\liminf_{n \to \infty} \int_{0}^1
	d_vG(x_n,\dot{x}_n)[\dot{V}_n]ds
	\ge 0.
	\end{equation}
	Let $[a,b]\subset [0,1]$ be such that $x([a,b])$ is in a single chart $(U,\varphi)$, with $U \subset \R^N$ and $\varphi:U \to \M$.
	If $n$ is big enough, also $x_n([a,b]) \subset U$.
	Considering the restriction of $x$ and $x_n$ on $[a,b]$, we define $\xi(s) = \varphi^{-1}(x(s))$ and $\xi_n(s) = \varphi^{-1}(x_n(s))$.
	Using this local chart, 
	\begin{equation}
	\dot{V}_n = \dot{\xi} - \dot\xi_n + w_n,
	\end{equation}
	where $w_n^i$ is $L^2$ convergent to $0$.
	We define $\tilde{G}:U \times \R^N \to \R$ as
	\begin{equation}
	\tilde{G}(\zeta,\eta) = G(\varphi(\zeta),d\varphi(\zeta)[\eta]).
	\end{equation}
	From \eqref{eqn:PSproof2} we infer
	\begin{equation}
	\liminf_{n \to \infty}  \int_{a}^{b} \p_v \tilde{G}(\xi_n,\dot\xi_n)[\dot\xi - \dot\xi_n + w_n^i] ds \ge 0,
	\end{equation}
	and by the $L^2$ convergence of $w_n^i$ to zero we obtain
	\begin{equation}
	\liminf_{n \to \infty} \int_{a}^{b} \p_v \tilde{G}(\xi_n,\dot\xi_n)[\dot\xi - \dot\xi_n] ds \ge 0.
	\end{equation}
	Since $x_n$ weakly converges to $x$, then $\xi_n$ weakly converges to $\xi$ and we have
	\begin{equation}
	\lim_{n \to \infty}  \int_{a}^{b} \p_v \tilde{G}(\xi_n,\dot\xi)[\dot\xi - \dot\xi_n] ds = 0.
	\end{equation}
	As a consequence,
	\begin{equation}
	\label{eqn:PSproof3}
	\liminf_{n \to \infty}\int_{a}^{b} 
	\left( 
	\p_v \tilde{G}(\xi_n,\dot\xi_n) - 
	\p_v \tilde{G}(\xi_n,\dot\xi)
	\right) 
	[\dot\xi - \dot\xi_n] ds \ge 0
	\end{equation}
	Due to the lack of regularity of $G$ on the zero section, we cannot apply the mean value theorem on the whole interval $[a,b]$.
	Thus we define $\delta_n(s) = d\varphi(\xi_n(s))[\dot\xi_n(s)]$ and $\delta(s) = d\varphi(\xi_n(s))[\dot\xi(s)]$ (note that the differential is evaluated on $\xi_n(s)$ for both $\delta_n$ and $\delta$).
	Let $A_n \subset [a,b]$ be the support of the $L^2$ function $\norm{\delta_n(s)}$ and 
	$A \subset [a,b]$ be the one of $\norm{\delta(s)}$.
	Then set
	\begin{equation}
	\begin{split}
	B_n & = \left\{s \in A_n \cap A: \frac{\delta_n}{\norm{\delta_n}} = - \frac{\delta}{\norm{\delta}} \ a.e.   \right\}, \\
	C_n & = (A_n \cup A) \setminus B_n ,\\
	D_n & = [a,b] \setminus (A_n \cup A).
	\end{split}
	\end{equation}
	On $C_n$, the zero vector is not on the segment joining $\delta_n$ and $\delta$, so we can use the mean value theorem and from \eqref{eqn:PSproof3} we obtain
	\begin{equation}
	\liminf_{n \to \infty} \int_{C_n} \p^2_{vv} \tilde{G}(\xi_n,\sigma \dot{\xi}_n + (1-\sigma)\dot\xi)[\dot\xi_n - \dot\xi,\dot\xi - \dot\xi_n] ds \ge 0,
	\end{equation}
	where $\sigma:C_n \to [0,1]$.
	The previous equation can be written as
	\begin{equation}
	\limsup_{n \to \infty} \int_{C_n} \p^2_{vv} \tilde{G}(\xi_n,\sigma \dot{\xi}_n + (1-\sigma)\dot\xi)[\dot\xi- \dot\xi_n,\dot\xi - \dot\xi_n] ds \le 0.
	\end{equation}
	Since $G$ is a strictly fiberwise convex function, there exists a positive constant $K_1$, which depends also on the local chart, such that
	\begin{multline}
	\limsup_{n \to \infty} \int_{C_n} \lVert \dot\xi - \dot{\xi}_n\rVert^2 ds \\
	\le K_1 \limsup_{n \to \infty} \int_{C_n} \p^2_{vv} \tilde{G}(\xi_n,\sigma \dot{\xi}_n + (1-\sigma)\dot\xi)[\dot\xi- \dot\xi_n,\dot\xi - \dot\xi_n] ds \le 0,
	\end{multline}
	so
	\begin{equation}
	\label{eqn:convCn}
	\lim_{n \to \infty} \int_{C_n} \lVert \dot\xi - \dot{\xi}_n\rVert^2 ds  = 0.
	\end{equation}
	On $B_n$ we can define a function $\lambda_n$ such that $\delta = -\lambda_n \delta_n$.
	Recalling the definition of $\tilde{G}$, from \eqref{eqn:PSproof3} we obtain 
	\begin{multline}
	\liminf_{n \to \infty} \int_{B_n}
	\left(
	d_vG(\varphi(\xi_n),d\varphi(\xi_n)[\dot\xi_n]) -
	d_vG(\varphi(\xi_n),d\varphi(\xi_n)[\dot\xi])
	\right)
	[\delta - \delta_n]
	ds  \\
	=
	- \liminf_{n \to \infty} \int_{B_n}
	(1 + \lambda_n) d_vG(\varphi(\xi_n),\delta_n)[\delta_n] ds \\
	- \liminf_{n \to \infty} \int_{B_n}
	\left(1 + \frac{1}{\lambda_n}\right) d_vG(\varphi(\xi_n),\delta)[\delta] ds \ge 0.\\
	\end{multline}
	Since $G$ is fiberwise homogeneous of degree two, by Euler's theorem the previous equation can be written as
	\begin{equation}
	\limsup_{n \to \infty} \int_{B_n}
	\left(
	2 (1 + \lambda_n) G(\varphi(\xi_n),\delta_n) +
	2 \left(1 + \frac{1}{\lambda_n}\right)G(\varphi(\xi_n),\delta) 
	\right) ds \le 0.
	\end{equation}
	Using \eqref{eqn:boundG}, there exists a positive constant $K_2$, which depends also on the local charts, such that 
	\begin{multline}
	\limsup_{n \to \infty} \int_{B_n} \left(\norm{\delta_n}^2 + \norm{\delta}^2 \right) ds \\
	\le 
	K_2 \limsup_{n \to \infty} \left(
	2 (1 + \lambda_n) G(\varphi(\xi_n),\delta_n) +
	2 \left(1 + \frac{1}{\lambda_n}\right)G(\varphi(\xi_n),\delta) 
	\right) ds \le 0,
	\end{multline}
	so 
	\begin{equation}
	\lim_{n \to \infty} \int_{B_n} \left(\norm{\delta_n}^2 + \norm{\delta}^2 \right) ds = 0.
	\end{equation}
	Since $\varphi$ is a diffeomorphism, there exists a positive constant $K_3$ such that
	\begin{equation}
	\lVert\dot\xi - \dot\xi_n\rVert^2 
	\le K_3 \norm{d\varphi(\xi_n)[\dot\xi - \dot\xi_n]}^2 
	= K_3 \norm{\delta - \delta_n}^2.
	\end{equation}
	As a consequence,
	\begin{multline}
	\label{eqn:convBn}
	\lim_{n \to \infty} \int_{B_n}
	\lVert\dot\xi - \dot\xi_n\rVert^2 ds 
	\le K_3
	\lim_{n \to \infty} \int_{B_n}
	\norm{\delta - \delta_n}^2 ds\\
	\le K_3
	\lim_{n \to \infty} \int_{B_n}
	\left(\norm{\delta_n}^2 + \norm{\delta}^2 \right) ds = 0.
	\end{multline}
	Since both $\delta$ and $\delta_n$ are zero on $D_n$, we have
	\begin{equation}
	\label{eqn:convDn}
	\lim_{n \to \infty} \int_{D_n}
	\lVert\dot\xi - \dot\xi_n\rVert^2 ds 
	\le K_3
	\lim_{n \to \infty} \int_{D_n}
	\norm{\delta - \delta_n}^2 ds = 0
	\end{equation}
	Summing up \eqref{eqn:convCn}, \eqref{eqn:convBn} and \eqref{eqn:convDn} we obtain that
	\begin{equation}
	\lim_{n \to \infty} \int_{a}^b
	\lVert\dot\xi - \dot\xi_n\rVert^2 ds = 0,
	\end{equation}
	and since this holds on every local chart, $\dot\xi_n$ strongly converges to $\dot\xi$ in $L^2$.
	As a consequence, 
	$x_n$ converges strongly to $x$ in $H^{1,2}$.
\end{proof}

\section{The pseudo-gradient vector field}

In this section we prove the existence of a pseudo-gradient vector field for $\J$ on $\Mf$ that is in $\V^-(x,\Mf)$.
Moreover, the flow of such a vector field defines  homotopies for which the Ljusternik and Schnirelmann relative category
is invariant.
In this way, we can prove the existence and multiplicity of \OFGC s using a minimax argument.

We shall need the following notation:
for all $x \in \Mf$ and $r > 0$ we define
$$
B_{r}(x) = \left\{ x \in \Mf: \dist_*(x,x_i) \le r  \right\}
$$
and
$$
U_r(x) = B_r(x) \cup B_r(\mR x),
$$
where $\dist_*$ has been defined in \eqref{eqn:defDist*}.

\begin{definition}
	For a given $x \in \Mf$ and a constant $\mu > 0$, we say that $\J$ has 
	\textit{
		$\V^-$-steepness greater or equal than $\mu$ at $x$ 
	}
	if there exists $\xi \in \V^-(x,\Mf)$ such that
	\begin{itemize}
		\item[i.] 	$\norm{\xi}_* = 1$;
		\item[ii.]  $d\J(x)[\xi] \le -\mu$.
	\end{itemize}
	In this case $\xi$ is a \textit{direction of $\mu$-steep $\V^-$-descent for $\J$ at $x$.}
\end{definition}
We set
\begin{equation}
\mathcal{W}^-(x,\Mf) = 
\left\{
\xi \in \V^-(x,\Mf):
\sprod{\xi(s)}{\nabla\Phi(x(s))} < 0
\text{ if } \Phi(x(s)) = 0, s \in ]0,1[
\right\}.
\end{equation}

\begin{proposition}
	\label{prop:localVecField}
	Set $\mu > 0$.
	If $\J$ has 
	$\V^-$-steepness greater or equal than $\mu$ at $x$, 
	then for any $\epsilon > 0$ there exists a $\rho > 0$ and a $C^1$ vector field $V$ defined in $B_\rho(x)$ such that:
	\begin{itemize}
		\item[(1)] $V(y) \in \mathcal{W}^-(y,\Mf)$,
		\item[(2)] $\norm{V(y)}_* = 1$,
		\item[(3)] $d\J(y)[V(y)] \le -\mu + \epsilon$,
	\end{itemize}
	for all $y \in B_\rho(x)$.
	Moreover, if $F$ is reversible, then $V$ can be defined in $U_\rho(x)$ and it also satisfies
	\begin{itemize}
		\item[(4)] 	$\overline\mR V(y) = V(\mR y)$
	\end{itemize}
	for all $ y \in U_\rho(x)$.
\end{proposition}
The proof of this proposition can be found in \cite[Proposition 4.3]{Giambo2018}.

\begin{proposition}
	\label{prop:vectorFieldW}
	Let $C \subset \Mf$ be closed set that does not contain any $\V^-$-critical curve for $\J$ and such that $\J(C) \subset [\d_m, \delta_M]$, where $\d_m$ and $\d_M$ are defined in \eqref{eqn:defBarDelta} and \eqref{eqn:deltaM} respectively.
	Then there exists a constant $\mu_C > 0$ and a locally Lipschitz continuous vector field $W$ defined on $C$ such that
	\begin{itemize}
		\item[(1)] $W(x) \in \mathcal{W}^-(x,\Mf)$,
		\item[(2)] $\norm{W(x)}_* \le 1$,
		\item[(3)] $d\J(x)[W(x)] \le -\mu_C$,
	\end{itemize}
	for all $x \in C$.
	Moreover, if $F$ is reversible and $C$ is $\mR$-invariant, then $W$ also satisfies
	\begin{itemize}
		\item[(4)] $\overline\mR W(x) = W(\mR x)$
	\end{itemize}
\end{proposition}
\begin{proof}
	There exists a constant $\mu_C > 0$ such that $\J$ has $\V^-$-steepness equal or greater than $2\mu_C$ at every $x \in C$.
	In fact, if such a constant does not exist, then there exists a Palais-Smale sequence $(x_n)_n \subset C$ at level $c \in [\delta_m,\delta_M]$.
	By Proposition \ref{prop:PS}, there exists a $\V^-$-critical curve for $\J$ in $C$, which is an absurd.
	
	Then, setting $\epsilon = \mu_C$, for all $x \in C$ we can take $\rho_x$ as in Proposition \ref{prop:localVecField} and a vector field $V_x$ defined in $B_{\rho_x}(x)$ satisfying (1)-(3) of 
	Proposition \ref{prop:localVecField}, thus
	$$
	d\J(y)[V(y)] \le -\mu_C,
	\qquad 
	\text{for all } y \in B_{\rho_x}(x).
	$$
	Consider the open covering 
	$
	\left\{
	B_{\rho_x}(x)
	\right\}_{x \in C}
	$
	of $C$.
	Because $C$ is a metric space, it is paracompact and there exists a locally finite covering 
	$
	\left\{B_{\rho_{x_i}}(x_i)\right\}_{i \in J}
	$
	of $C$.
	For all $y \in C$ and $i \in J$, we define
	$$
	\varrho_i(y) = \dist_*(y, C \setminus B_{\rho_{x_i}}(x_i) )
	$$
	and
	$$
	\beta_i(y) = \frac{\varrho_i(y)}{\sum_{i \in J} \varrho_i(y) }.
	$$
	Then $\sum_{i \in J} \beta_i(y) = 1$ for all $y \in C$ and we define our desired vector field as
	\begin{equation}
	W(y) = \sum_{i \in J} \beta_i(y) V_{x_i}(y).
	\end{equation}
	If $F$ is reversible and $C$ is $\mR$-invariant,
	Proposition \ref{prop:localVecField} ensures that for all $x \in C$ there exists $\rho_x > 0$ a vector field defined in $U_{\rho_x}(x)$ such that 
	$$
	d\J(y)[V(y)] \le -\mu_C
	\quad \text{and} \quad
	\overline\mR V(y) = V(\mR y),
	\qquad 
	\text{for all } y \in U_{\rho_x}(x).
	$$
	Then we can consider a locally finite covering 
	$
	\left\{U_{\rho_{x_i}}(x_i)\right\}_{i \in J}
	$
	of $C$, and defining 
	$$
	\varrho_i(y) = \dist_*(y, C \setminus U_{\rho_{x_i}}(x_i) )
	$$
	we obtain our desired vector field following the previous construction.
\end{proof}

\section{Deformation Lemmas}
\begin{definition}
	A number $c > 0$ is \textit{a $\V^-$-critical value for $\J$ on $\Mf$ }if there exists $x \in \Mf$ that is a $\V^-$-critical curve for $\J$ on $\Mf$ such that $\J(x) = c$.
	Otherwise, $c$ is said \textit{$\V^-$-regular value for $\J$ on $\Mf$}.
\end{definition}
For any $c >0$, we set
\begin{equation}
\J^c =
\left\{
x \in \Mf: \J(x) \le c
\right\}
\end{equation}

\begin{definition}
	Let $\mathcal{N}$ be a subset of $\Mf$.
	Then a continuous function $h:[0,1]\times \mathcal N \to \mathcal N$ is said \textit{admissible homotopy} if
	\begin{itemize}
		\item[i)] $h(0,x) = x$ for all $x \in \mathcal{N}$;
		\item[ii)] $h(\tau,x) \in \Cf_0$ for all $x \in \mathcal{N}\cap \Cf_0$ and $\tau \in [0,1]$;
		\item[iii)] if $x \notin \mathcal{N}\cap \Cf_0$, then $h(\tau,x) \notin \mathcal{N}\cap \Cf_0$ , for all $\tau \in [0,1]$;
		\item[iv)] if $F$ is a reversible, $\mathcal{N}$ must be $\mR$-invariant and $h(\tau,\mR x) = \mR h(\tau,x)$ for every $\tau \in [0,1]$.
	\end{itemize}
\end{definition}

\begin{lemma}
	\label{lem:deformation1}
	Let $c > 0$ be a $\V^-$-regular value for $\J$ on $\Mf$.
	Then there exists an $\epsilon > 0$ and an admissible homotopy $h:[0,1] \times \J^{c+ \epsilon} \to \J^{c + \epsilon}$ such that
	$$
	h(1,\J^{c+ \epsilon}) \subset \J^{c - \epsilon}.
	$$
\end{lemma}
\begin{proof}
	By the Proposition \ref{prop:PS}, there exists $\bar\epsilon > 0$ such that $c - \bar\epsilon > 0$ and there are no $\V^-$-critical curves on
	$C = \J^{-1}([c - \bar\epsilon,c + \bar\epsilon])$.
	By Proposition \ref{prop:vectorFieldW}, there exists a vector field $W$ defined on $C$ and a constant $\mu_C > 0$ such that 
	$$
	d\J(x)[W(x)] \le - \mu_C,
	\qquad \forall x \in C.
	$$
	Let $\chi:\R^+ \to [0,1]$ be the piecewise affine function such that
	$$
	\chi(t) = 0, \quad \text{if } t \le c - \bar\epsilon
	\quad
	\text{and}
	\quad 
	\chi(t) = 1, \quad \text{if } t \ge c - \frac{\bar\epsilon}{2}.
	$$
	Then we define our desired homotopy $h$ as the solution of the Cauchy problem
	\begin{equation}
	\left\lbrace
	\begin{array}{l}
	\displaystyle\frac{\p h}{\p \tau}(\tau,x) = \chi\left(\J(h(\tau,x))\right)	W(h(\tau,x)) \\ 
	$ $  \\
	h(0,x)  = x
	\end{array}
	\right.
	\end{equation}
	if $\J(x) \in [c - \bar\epsilon, c + \bar\epsilon]$ and
	$$
	h(\tau, x) = x, \quad \forall \tau \in [0,1] \quad \text{if } \J(x) \le c - \bar\epsilon.
	$$
	We can ensure that $h(1,\J^{c + \epsilon}) \subset \J^{c- \epsilon}$ setting $\epsilon = \min\{\frac{\bar\epsilon}{2}, \frac{\mu_C}{2}\}$.
	Indeed, if $\epsilon \le \frac{\bar\epsilon}{2}$, then $\chi(t) = 1$ for any $t \in [c- \epsilon, c + \epsilon]$.
	Consequently, if $h (\tau, x) \in \J^{-1}([c-\epsilon,c + \epsilon])$ then
	$$
	d\J(h(\tau,x))\left[ \frac{\p h}{\p \tau}(\tau,x) \right]
	= 
	d\J(h(\tau,x))[W(h(\tau,x))] \le \mu_C,
	$$
	hence
	$$
	\J(h(\tau,x)) \le \J(h(0,x)) - \mu_C \tau \le c + \epsilon - \mu \tau.
	$$
	From the last inequality we infer that if $\epsilon \le \frac{\mu_C}{2}$, then $\J(h(1,x)) \le c - \epsilon$ for all $x \in \J^{-1}([c-\epsilon,c + \epsilon])$.
	
	The homotopy $h:h:[0,1] \times \J^{c+ \epsilon} \to \J^{c + \epsilon}$ defined above is an admissible homotopy.
	In fact, $h(0,x) = x$ for all $x \in \J^{c + \epsilon}$ and, if $x \in \Cf_0$ then $\J(x) = 0$ and $h(\tau,x) \in \Cf_0$ for all $\tau \in [0,1]$.
	Furthermore, $h$ does not move the curves in $\J^{c - \bar\epsilon}$, thus if $x \notin \J^{c + \epsilon }\cap \Cf_0$, then $h(\tau,x) \notin \J^{c + \epsilon }\cap \Cf_0$ for all $\tau \in [0,1]$.
	Finally, if $F$ is a reversible Finsler metric, then $\J^{c + \epsilon}$ is $\mR$-invariant and $h(\tau,\mR x) = \mR h(\tau, x)$ for all $x \in \J^{c + \epsilon}$ and $\tau \in [0,1]$, since $\overline\mR W(x) = W(\overline{\mR} x)$ for all $x \in \J^{c + \epsilon} \subset \Mf$ by Proposition \ref{prop:vectorFieldW}.
\end{proof}

From now on, let us assume that the number of non-constant $\V^-$-critical curves for $\J$ on $\Mf$ is finite. Otherwise, by Proposition \ref{prop:V-crit-OFG}, Theorem \ref{teo:mainTeo} is trivially true.

Let us fix an $r_* > 0$ such that
\begin{itemize} 
	\item $\overline{U_{r_*}(x_i)} \cap \overline{U_{r_*}(x_j)} = \emptyset$ for every $i\ne j$;
	\item any $\overline{U_{r_*}(x_i)}$ is contractible in itself; 
	\item any $\overline{U_{r_*}(x_i)}$ does not include constant curves.
\end{itemize}
Thus, we define
\begin{equation}
\mathcal{O}^* = \bigcup_{i = 1,\dots,k} 
U_{r^*}(x_i).
\end{equation}
We remark that, if $F$ is reversible, then $\mathcal{O}^*$ is $\mR$-invariant by Lemma \ref{lem:RxCritical}.

Using the same procedure exploited in the proof of Lemma \ref{lem:deformation1}, we obtain the following result.
\begin{lemma}
	\label{lem:deformation2}
	Assume that the number of non-constant $\V^-$-critical curves for $\J$ on $\Mf$ is finite and let $c > 0$ be a $\V^-$-critical value for $\J$ on $\Mf$.
	Then there exists an $\epsilon > 0$ and an admissible homotopy $h:[0,1] \times \J^{c+ \epsilon} \to \J^{c + \epsilon}$ such that
	$
	h(1,\J^{c+ \epsilon} \setminus \mathcal{O}^*) \subset \J^{c - \epsilon}.
	$
\end{lemma}

\section{The minimax principle}
In this section, we assume that $F$ is a reversible Finsler metric.
Otherwise, all the results can be applied with slight modifications that we will show later.
Let $\mathcal{D}$ be the set of all closed $\mR$-invariant subsets of $\Mf$.
By \eqref{eqn:relCatReversible}, we have that
\begin{equation}
\cat_{\tilde\Mf,\tilde\Cf_0}\tilde\Mf \ge 
\cat_{\tilde\Cf,\tilde\Cf_0} \tilde\Cf \ge N,
\end{equation}
so the sets
\begin{equation}
\label{eqn:defGamma_i}
\Gamma_i = \left\{
D \in \mathcal{D}: \cat_{\tilde\Mf,\tilde\Cf_0}\tilde D \ge i
\right\}
, \qquad i = 1, \dots, N
\end{equation}
are all non-empty and the following quantity is well defined
\begin{equation}
\label{eqn:cDef}
c_i =
\inf_{D\in \Gamma_i}
\sup_{x \in D}
\J(x),
\quad \text{for any $i = 1, \dots, N$}.
\end{equation}
\begin{lemma}
	\label{lem:criticalValuesCi}
	The following statements hold:
	\begin{itemize}
		\item[(1)] $c_1 \ge \d_m$, where $\d_m$ is defined in 
		\eqref{eqn:defBarDelta}.
		\item[(2)] for every $i = 1, \dots, N$, $c_i$ is a $\V^-$-critical value for $\J$ on $\Mf$
		\item[(3)] for every $i = 1, \dots, N-1$, $c_i < c_{i + 1}$.
	\end{itemize}	
\end{lemma}
\begin{proof}
	By \eqref{eqn:delta0Def}, there exists a exists a $C^1$--retraction 
	\begin{equation}
	\label{eqn:rDef}
	\mathbf{r}: \Phi^{-1}([-\delta_0,\delta_0]) \to \p\Omega
	\end{equation}
	defined in terms of the flow of $\nabla\Phi$.
	By Lemma \ref{lem:barDelta}, if $x \in \J^{\d_m}$ then $\Phi(x(s)) \in [0,-\delta_0]$.
	Using the retraction \eqref{eqn:rDef}, there exists a homotopy $h$ such that
	$h(1,x)(s) \in \p\Omega$.
	Now define the homotopy
	\[
	k(\tau,x)(s)=x\big((1-\tau)s+\tfrac\tau2\big),
	\]
	so that $k(1,x)(s) = x(1/2) \in \p\Omega$ for every curve that lies in $\p\Omega$.
	Combining the two homotopies $h$ and $k$ we define 
	$$
	H(\tau,x) = 
	\left\lbrace\begin{array}{ll}
	h(2\tau,x) & \text{if } \tau \in\left[0,\frac{1}{2}\right] \\
	k(2\tau - 1, h(1,x)) & \text{if } \tau \in \left[\frac{1}{2},1\right]
	\end{array}
	\right.
	,
	$$
	which is an $\mR$-invariant homotopy such that
	$$
	H(1,\J^{\d_m}) \in \Cf_0
	$$
	and, consequently,
	$$
	\cat_{\widetilde\Mf,\widetilde\Cf_0} \widetilde{\J^{\d_m}} = 0.
	$$
	As a consequence, if $\cat_{\tilde\Mf,\tilde\Cf_0} \tilde{D} \ge 1$, then $D \not\subset \J^{\d_m}$ and there exists a $x \in D$ such that $\J(x) > \d_m$.
	By definition of $c_1$, we infer that $c_1 \ge \d_m$.
	
	Let us prove statement (2).
	Seeking for a contradiction, let $c_i$ be a $\V^-$-regular value.
	By definition of $c_i$, for all $\epsilon > 0$ there exists a $D \in \Gamma_i$ such that
	$$
	\J(x) \le c_i + \epsilon, \qquad \forall x \in D.
	$$
	so $D \subset \J^{c_i + \epsilon}$.
	By Lemma \ref{lem:deformation1}, there exists an admissible homotopy $h: [0,1] \times \J^{c_i + \epsilon} \to \J^{c_i + \epsilon}$ such that 
	$h(1,D) \subset \J^{c_i - \epsilon}$.
	Since the relative category is invariant by admissible homotopies, 
	\begin{equation}
	\cat_{\tilde\Mf,\tilde\Cf_0} \widetilde{h(1,D)} = 
	\cat_{\tilde\Mf,\tilde\Cf_0} \widetilde D \ge i.
	\end{equation}
	Then 
	\begin{equation}
	c_i \le \sup_{x \in D}\J(h(1,x)) \le c_i - \epsilon,
	\end{equation}
	which is an absurd.
	
	We also prove statement (3) arguing by contradiction.
	Let $i$ be such that $c = c_i = c_{i+1}$.
	Since $c$ is a $\V^-$-critical value, we can apply Lemma \ref{lem:deformation2} and there exist an $\epsilon > 0$ and an admissible homotopy $h:[0,1] \times \J^{c + \epsilon} \to \J^{c + \epsilon}$ such that 
	$$
	h(1,\J^{c+ \epsilon} \setminus \mathcal{O}^*) \subset \J^{c - \epsilon},
	$$
	where $\widetilde{\mathcal{O}}^*$ is a contractible set n $\widetilde{\Mf}\setminus\widetilde{\Cf}_0$.
	Moreover, there exists a $D \in \Gamma_{i+1}$ such that $D \subset \J^{c+ \epsilon}$, so 
	\begin{multline}
	i + 1 \le 
	\cat_{\tilde\Mf,\tilde\Cf_0} \widetilde{\J^{c+ \epsilon}} \le
	\cat_{\tilde\Mf,\tilde\Cf_0} \left(\widetilde{\J^{c+ \epsilon}} \setminus \widetilde{\mathcal{O}}^* \right) + 
	\cat_{\tilde\Mf,\tilde\Cf_0} \widetilde{\mathcal{O}}^*
	\\
	\le
	\cat_{\tilde\Mf,\tilde\Cf_0} \left(\widetilde{\J^{c+ \epsilon}} \setminus \widetilde{\mathcal{O}}^* \right)  + 1
	\le 
	\cat_{\tilde\Mf,\tilde\Cf_0} \left(\widetilde{\J^{c-  \epsilon}}  \right)  + 1
	< i + 1,
	\end{multline}
	that is a contradiction.
\end{proof}

\section{Proof of the main theorem}
Now, we are ready to collect the previous results and give the proof of the main theorem.
\begin{proof}[Proof of Theorem \ref{teo:mainTeo}]
	Let the number of $\V^-$-critical curves for $\J$ on $\Mf$ be finite, otherwise the theorem derives directly from Proposition \ref{prop:V-crit-OFG}.
	If $F$ is a reversible Finsler metric, by Lemma \ref{lem:criticalValuesCi} there are at least $N$ $\V^-$-critical values that are strictly greater than zero, so they correspond to non-constant $\V^-$-critical curves.
	By Proposition \ref{prop:V-crit-OFG}, either exists an \OTFGC\ or they are \OGC s.
	It remains to prove that two different $\V^-$-critical values have geometrically distinct $\V^-$-critical curves, and it is sufficient to prove that if $x_1$ and $x_2$ are non-constant $\V^-$-critical curves such that $x_1([0,1]) = x_2([0,1])$, then $\J(x_1) = \J(x_2)$.
	Since $x_1$ and $x_2$ are non-constant $\V^-$ critical curves, $\dot{x}_1$ and $\dot{x}_2$ are always different from zero and there exists a function $\theta:[0,1] \to [0,1]$ of class $C^2$ such that
	$$
	x_2(s) = x_1(\theta(s)),
	$$
	and $\dot\theta(s) \ne 0$ for any $s \in [0,1]$.
	Moreover, either $\theta(0) = 0$ and $\theta(1) = 1$ or 
	$\theta(0) = 1$ and $\theta(1) = 0$.
	Recalling that both $x_1$ and $x_2$ satisfy the geodesics equations \eqref{eqn:geodesicEquationsChern} and that the components $\Gamma_{jk}^i$ of the Chern connection are fiberwise homogeneous of degree zero when the Finsler metric is reversible, we obtain that
	\begin{multline}
	\label{eqn:geomDistinction}
	0 = \ddot{x}_2^i + \Gamma_{jk}^i(x_2,\dot{x}_2)\dot{x}_2^j\dot{x}_2^k \\
	= 
	\ddot{\theta}\ \dot{x}_1^i(\theta) +
	\dot{\theta}^2\ \ddot{x}_1^i(\theta) +
	\dot{\theta}^2\ 
	\Gamma_{jk}^i\left(x_1(\theta),\dot\theta\ \dot{x}_1(\theta)    \right) \dot{x}_1^j(\theta)\dot{x}_1^k(\theta)
	\\
	=
	\ddot{\theta}\ \dot{x}_1^i(\theta) +
	\dot{\theta}^2\ 
	\left(
	\ddot{x}_1^i(\theta)  +
	\Gamma_{jk}^i \left( x_1(\theta), \dot{x}_1(\theta) \right) \dot{x}_1^j(\theta)\dot{x}_1^k(\theta)
	\right)
	= 
	\ddot{\theta}\ \dot{x}_1^i(\theta).
	\end{multline}
	for every $i = 1,\dots, N$.
	As a consequence, $\ddot{\theta}(s) = 0$ for all $s \in [0,1]$, thus either $\theta(s) = s$ or $\theta(s) = 1 - s$.
	In both cases, we obtain that
	$$
	\J(x_1) = 
	\frac{1}{2}\int_{0}^{1}	G(x_1,\dot{x}_1)\ ds
	=
	\frac{1}{2}\int_{0}^{1}	G(x_2,\dot{x}_2)\ ds
	= 
	\J(x_2).
	$$
	
	If $F$ is not reversible, we exploit the relative category
	$$
	\cat_{\Mf,\Cf_0} \Mf \ge \cat_{\Cf,\Cf_0} \Cf \ge 2,
	$$
	and we replace the definition of $\Gamma_i$ in \eqref{eqn:defGamma_i} with
	\begin{equation}
	\Gamma_i = 
	\left\{
	D \in \mathcal{D}: \cat_{\Mf,\Cf_0} D \ge i
	\right\},
	\qquad
	i = 1,2.
	\end{equation}
	Consequently, defining the numbers $c_i$ as in \ref{eqn:cDef}, Lemma \ref{lem:criticalValuesCi} holds and there are two non constant $\V^-$-critical curves for $\J$ on $\Mf$ with different values of the energy functional.
	By Proposition \ref{prop:V-crit-OFG}, either there exists an \OTFGC\ or there are two \OFGC\ with different values of the energy functional.
\end{proof} 

\begin{remark}
	If $F$ is not reversible, the geometrical distinction of two \OFGC s with different values of the energy functional cannot be ensured in general.
	For example, set 
	$
	\Omega = \left\{ q \in \R^2: \norm{q} < 1/2	\right\}
	$
	and consider the Randers metric
	$
	F(q,v) = \norm{v} + \beta v^1
	$,
	with $\beta \in (0,1)$.
	Then  
	$$
	x_1(s) = \left(s - \frac{1}{2}, 0 \right)
	\quad \text{and} \quad
	x_2(s) = \left(\frac{1}{2} - s, 0 \right),
	\qquad s \in [0,1],
	$$
	are both \OFGC s and an easy computation shows that 
	$$
	\J(x_1) = \left(1 + \beta \right)^2  \ne \left(1 - \beta \right)^2  = \J(x_2).
	$$
	However, $x_1$ and $x_2$ are not geometrically distinct.
\end{remark}



\providecommand{\href}[2]{#2}
\providecommand{\arxiv}[1]{\href{http://arxiv.org/abs/#1}{arXiv:#1}}
\providecommand{\url}[1]{\texttt{#1}}
\providecommand{\urlprefix}{URL }

\end{document}